\documentclass[preprint,12pt]{elsarticle}

\usepackage{amssymb}
\usepackage{amsmath, amsthm} %amsfonts,amsthm 
\usepackage{color}
\usepackage[dvipsnames]{xcolor} % more colors: forestgreen
\usepackage{float}
\usepackage{dsfont}
\newcommand{\C}{{\mathcal{C}}} % copulas
\newcommand{\Csym}{{\mathcal{C}^\text{sym}}} % symmetric copulas
\newcommand{\R}{{\mathbb{R}}} % real numbers
\newcommand{\N}{{\mathbb{N}}} % natural numbers
\newcommand{\B}{\mathcal{B}} % Borelsche Sigma-Algebra
\DeclareMathOperator{\EX}{\mathbb{E}}% expected value
\newcommand{\D}{{\mathcal{D}}} % set of diagonals
\newcommand{\rom}[1]{\uppercase\expandafter{\romannumeral#1}} % rom. zahlen
\newtheorem{Theorem}{Theorem}[section]
\newtheorem{Definition}[Theorem]{Definition} 
\newtheorem{Lemma}[Theorem]{Lemma}	
\newtheorem{Example}[Theorem]{Example}
\newtheorem{Remark}[Theorem]{Remark} %\textbf{Remark}
\newtheorem{Proposition}[Theorem]{Proposition}

\usepackage{tikz}
\usetikzlibrary{positioning}
\usepackage{tikz-cd}
\usepackage{tkz-euclide}
\usepackage{subcaption}
\usepackage{hyperref}

\makeatletter
\def\ps@pprintTitle{%
  \let\@oddhead\@empty
  \let\@evenhead\@empty
  \def\@oddfoot{\reset@font\hfil\thepage\hfil}
  \let\@evenfoot\@oddfoot
}
\makeatother

\journal{}

\begin{document}

\begin{frontmatter}



\title{Revisiting the region determined by Spearman's $\rho$ and Spearman's footrule $\phi$}


\author[UniSbg]{Marco Tschimpke} \author[AMAG]{Manuela Schreyer} \author[UniSbg]{Wolfgang Trutschnig \corref{mycorrespondingauthor}}

\affiliation[UniSbg]{organization={University of Salzburg, Department for Artificial Intelligence \& Human Interfaces},%Department and Organization
            addressline={Hellbrunnerstrasse 34}, 
            city={Salzburg},
            postcode={5020}, 
            state={Salzburg},
            country={Austria}}
\affiliation[AMAG]{organization={AMAG Austria Metall GmbH},%Department and Organization
            addressline={Lamprechtshausenerstr 61}, 
            city={Ranshofen},
            postcode={5282}, 
            state={Upper Austria},
            country={Austria}}
\cortext[mycorrespondingauthor]{Corresponding author. Email address: wolfgang@trutschnig.net}

\begin{abstract}
Kokol and Stopar ($2023$) recently studied the exact region $\Omega_{\phi,\rho}$
determined by Spearman's footrule $\phi$ and Spearman's $\rho$ and derived a sharp lower, as well as a non-sharp 
upper bound for $\rho$ given $\phi$. Considering that the proofs for establishing these inequalities are 
novel and interesting, but technically quite involved we here provide alternative simpler proofs mainly building upon 
shuffles, symmetry, denseness and mass shifting. As a by-product of these proofs we derive several additional 
results on shuffle rearrangements and the interplay between diagonal copulas and shuffles which are of independent interest. 
Moreover we finally show that we can get closer to the (non-sharp) 
upper bound than established in the literature so far.   

\end{abstract}

\begin{keyword}
Copulas \sep Concordance \sep Shuffle \sep Markov kernel \sep Optimization 


\MSC 62H20 \sep 62H05 \sep 60J35
\end{keyword}

\end{frontmatter}


\section{Introduction}

A standard approach for quantifying the extent of concordance or, more generally, association of a pair $(X,Y)$ of random variables $X,Y$ is to consider different measures of (weak) concordance or association such as Spearman's $\rho$, Kendall's $\tau$, Gini's $\gamma$, Spearman's footrule $\phi$ or Blomqvist's $\beta$ 
(see \cite{scarsini1984measures, nelsen2007introduction}). 
Each of the just mentioned measures only depends on the dependence structure of $(X,Y)$, so in the case of continuous 
marginals all these measures are functions of the (unique) copula $C$ underlying $(X,Y)$. 
Given two measures of (weak) concordance $\kappa_1 $ and $\kappa_2$ a seemingly natural question 
is, how much the value of $\kappa_2$ can vary given the value of $\kappa_1$, or vice versa. 
In other words: one might naturally be interested in determining the region 
$$
\Omega_{\kappa_1,\kappa_2}:=\{(\kappa_1(C), \kappa_2(C)): C \in \C \},
$$
where $\C$ denotes the family of all bivariate copulas. The larger the portion of the rectangle 
$$\{\kappa_1(C): C \in \C \} \times \{\kappa_2(C): C \in \C \}$$ covered by $\Omega_{\kappa_1,\kappa_2}$, the more
different the measures $\kappa_1,\kappa_2$ may be considered.

The presumably most well known question in this context was, whether the inequality for Kendall's $\tau$ and 
Spearman's $\rho$ as established by Durbin and Stuart in \cite{durbin1951inversions} is sharp. 
This very question and some related ones were ans\-wered in \cite{schreyer2017exact}, where $\Omega_{\tau,\rho}$
was characterized and shown to be compact but not convex. 
Since then, various contributions have followed: Article \cite{kokol2023exact} studies the interrelations between Kendall's 
$\tau$ and Gini's $\gamma$ / Spearman's footrule $\phi$. The lower and upper bound for Spearman's 
$\rho$ (Gini's $\gamma$) given Spearman's footrule $\phi$ were established in 
\cite{bukovvsek2024exact} (\cite{bukovvsek2022exact}). Finally, \cite{nelsen2007introduction, bukovvsek2021spearman} cover 
the relations between Blomqvist's $\beta$ and all remaining measures of (weak) concordance (see Table \ref{TableIntro} 
for a quick overview). 
%The region determined by Gini's $\gamma$ and Spearman's footrule $\phi$ is given in \cite{bukovvsek2022exact}. Additionally, \cite{bukovvsek2024exact} analysed the region of Spearman's $\rho$ and Spearman's footrule $\phi$. Finally, \cite{nelsen2007introduction, bukovvsek2021spearman} provide the relation between Blomqvist's $\beta$ and all remaining measures of (weak) concordance (see Table \ref{TableIntro} for an overview). 

\begin{table}[H]
    \centering
    \begin{tabular}{|c| c c c c|} 
         \hline
          & $\tau$ & $\gamma$ & $\phi$ & $\beta$ \\ %[0.5ex] 
         \hline %\hline
         $\rho$ & \textcolor{ForestGreen}{\checkmark} & \textcolor{red}{$\times$} & \textcolor{orange}{$-$} & \textcolor{ForestGreen}{\checkmark} \\ 
         %\hline
         $\tau$ &  & \textcolor{ForestGreen}{\checkmark} & \textcolor{ForestGreen}{\checkmark} & \textcolor{ForestGreen}{\checkmark} \\ 
         %\hline
         $\gamma$ &  &  & \textcolor{ForestGreen}{\checkmark} & \textcolor{ForestGreen}{\checkmark} \\ 
         %\hline
         $\phi$ &  &  &  & \textcolor{ForestGreen}{\checkmark} \\  [1ex] 
         \hline
    \end{tabular}
    \caption{Already studied pairs of measures of (weak) concordance, with the following nomenclature: 
    \textcolor{ForestGreen}{\checkmark}...exactly known region, \textcolor{orange}{$-$}...partially known region, 
    \textcolor{red}{$\times$}...unknown region.}
    \label{TableIntro}
\end{table}
%$\rho$ &  & 2 & 0 & 1 & 2 \\ %\hline $\tau$ &  &  & 2 & 2 & 2 \\ %\hline $\gamma$ &  &  &  & 2 & 2 \\ %\hline $\phi$ &  &  &  &  & 2 \\  [1ex] 

In what follows we focus on $\Omega_{\phi,\rho}$, the region determined by Spearman's footrule 
$\phi$ and  Spearman's $\rho$. According to \cite{bukovvsek2024exact} the inequality
\begin{align}\label{PhiRhoRegionEq}
    \frac{2}{9} \sqrt{3} \left( 1 + 2\phi(C) \right)^{3/2} - 1 \leq \rho(C) \leq 1 - \frac{2}{3} \left(\phi(C) - 1 \right)^2
\end{align}
holds for every copula $C$. Furthermore (again see \cite{bukovvsek2024exact}), the lower bound in 
ineq. \eqref{PhiRhoRegionEq} is sharp while the upper bound is only known to be sharp in countably many points 
(with only accumulation point $(1,1)$). Main objective of our contribution is to show that both, the lower and 
the upper inequality in ineq. \eqref{PhiRhoRegionEq} can be established alternatively by proceeding similarly as in 
\cite{schreyer2017exact}. In fact, working with 
shuffles, symmetry and continuity, proving the right hand-side of \eqref{PhiRhoRegionEq} boils down 
to a straightforward application of classical Cauchy-Schwarz inequality 
(which, in turn, even provides a simple characterization for those shuffles, 
for which the inequality becomes an equality); and proving the lower (sharp) bound to a rearrangement 
property of integrals/sums (which holds in a very general setting). 
Apart from providing alternative simple proofs for ineq. \eqref{PhiRhoRegionEq} we also show that we can 
get closer to the upper bound than established in \cite{bukovvsek2024exact}.   

The rest of the contribution is organized as follows: Section \ref{SecNotation} introduces the necessary notation 
and preliminaries used in the sequel. Section \ref{SecUpperBound} provides two alternative simple proofs 
for the upper bound: one working with symmetric shuffles and the other one building upon maximality of 
diagonal copulas within the family of all symmetric copulas with given diagonal, and the interplay 
between diagonal copulas and symmetric shuffles. 
Section \ref{SecLowerBound} revisits the lower sharp inequality and derives it via the afore-mentioned 
novel rearrangement idea.
Finally, working with ordinal sums and `interpolations' of copulas,  
Section \ref{SecImprovementUpperBound} extends the known subset of $\Omega_{\phi,\rho}$.  
Se\-veral examples and graphics illustrate the main ideas and chosen approaches.

\section{Notation and preliminaries}\label{SecNotation}

Given an arbitrary metric space $(S,d)$, the Borel $\sigma$-field on $S$ will be denoted by $\B(S)$. 
Moreover, the one- and two dimensional Lebesgue measure (on $\mathcal{B}(\mathbb{R})$ and 
$\mathcal{B}(\mathbb{R}^2)$, respectively) will be denoted by $\lambda$ and $\lambda_2$,  
the Dirac measure in a point $a$ by $\Xi_a$. The class of all bivariate copulas is denoted by $\C$,
for every $C \in \C$ we will let $\mu_C$ denote the corresponding doubly stochastic 
measure. Prominent examples of copulas are the independent copula $\Pi$ and the lower and upper Fr\`echet-Hoeffding 
bounds $W$ and $M$. 
In the following $C^t$ will denote the transpose of a copula $C$, i.e., $C^t(u,v) = C(v,u)$ for all $u,v\in[0,1]$. 
A copula $C$ is called symmetric if $C = C^t$ holds. The uniform metric $d_\infty$ on $\C$ is defined by
$$
d_\infty(C, D) :=  \max_{(u,v) \in [0,1]^2} \vert C(u,v) - D(u,v) \vert.
$$
It is well known that $(\C, d_\infty)$ is a compact metric space (see \cite{durante2015principles}). 
For every measurable function 
$f:\R \to \R$ we set $\Vert f \Vert_\infty := \sup_{x \in \R} \vert f(x) \vert $.
For background on copulas and doubly stochastic measures we refer to 
\cite{nelsen2007introduction, durante2015principles}. 

A mapping $K: \R \times \B(\R) \to [0,1]$ is a called a Markov kernel if $x \mapsto K(x, F)$ is measurable for every set $F \in \B(\R)$ and $F \mapsto K(x,F)$ is a probability measure for every $x\in\R$. Given two random variables 
$X$ and $Y$ on a probability space $(\Omega, \B(\Omega), \mathbb{P})$ a Markov kernel is called regular conditional 
distribution of $Y$ given $X$ if
\begin{align*}
K(X(\omega), F) (\omega) = \EX\left( \mathds{1}_F \circ Y \vert X \right)(\omega)
\end{align*}
holds for every $F \in \B(\R)$ and $\mathbb{P}$-almost every $\omega \in \Omega$. It is well known that for every pair $(X,Y)$ 
the Markov kernel $K(x, \cdot)$ is unique for $\mathbb{P}^X$-almost every $x \in \mathbb{R}$. 
In the sequel we will write $(U,V) \sim C$ if $C$ is the distribution function of $(U,V)$ and $U,V$ are uniform on $[0,1]$. 
For every $C \in \C$ there exists a Markov kernel $K_C$ satisfying the disintegration property
\begin{align*}
    \mu_C(G) = \int_{[0,1]} K_C(u, G_u) d\lambda(u)
\end{align*}
for every $G\in \B([0,1]^2)$, where $G_u := \{ v \in [0,1]: (u,v) \in G \}$. For more information on disintegration 
and conditional expectations see \cite{kallenberg2002, klenke2008}; for more background on  
Markov kernels and their applications in the context of copulas  we refer to 
\cite{trutschnig2011strong, kasper2021weak, sfx2021vine}.

A measurable transformation $h: [0,1] \to [0,1]$ is called $\lambda$-preserving 
if $\lambda^h(E):=\lambda(h^{-1}(E)) = \lambda(E)$ holds for every $E \in \B([0,1])$, i.e., 
if the push forward $\lambda^h$ of $\lambda$ via $h$ coincides with $\lambda$.
A copula $C$ is said to be completely dependent if there exists a $\lambda$-preserving transformation 
$h: [0,1] \to [0,1]$ such that 
$K_C(x, F) = \mathds{1}_F (h(x))$ is a version of the Markov kernel of $C$. In other words: A copula
is called completely dependent if it allows a Markov kernel whose conditional distributions are all 
degenerated. For alternative equivalent defi\-nitions of complete dependence we refer to \citep{trutschnig2011strong}
and the references therein. 
For every $\lambda$-preserving transformation $h: [0,1] \rightarrow [0,1]$ we will let 
$C_h$ denote the corresponding (unique) copula and write $\C_d$ for the class of all 
completely dependent copulas. If a copula $C$ fulfills $C \in \C_d$ and $C^t \in \C_d$
we will refer to it as mutually completely dependent. It is straightforward to verify that 
the latter is the case if, and only if the corresponding $\lambda$-preserving transformation is 
bijective outside a set of $\lambda$-measure zero.

We call a $\lambda$-preserving transformation $h:[0,1] \to [0,1]$ a (classical) 
equidistant even shuffle (a.k.a. equidistant shuffle of $M$) with $N \in \mathbb{N}$ stripes if, and only if 
$h$ is linear with slope $1$ on each 
interval $I_N^i:=(\frac{i-1}{N}, \frac{i}{N})$, injective on $\bigcup_{i=1}^N (\frac{i-1}{N}, \frac{i}{N})$, and 
just permutes the intervals $I_N^1,\ldots,I_N^N$. In 
the sequel $\Sigma_N$ will denote the set of all permutations of the set $\{1, \dots, N\}$ and
$\mathcal{S}_N$ the family of all equidistant even shuffles with $N$ stripes.  
It is well-known and straightforward to check that for every $N \in \mathbb{N}$ 
there is a one-to-one correspondence between $\Sigma_N$ and $\mathcal{S}_N$ (see \cite{mikusinski1992shuffles}).
Emphasizing the permutation we will therefore frequently write $S_\pi \in \mathcal{S}_N$. 

In what follows we will work with the subclass of symmetric shuffles: A shuffle $S_\pi \in \mathcal{S}_N$ is called 
symmetric if, and only if the corresponding completely 
dependent copula $C_{S_\pi} \in \mathcal{C}_d$ is symmetric. It is straightforward to verify that 
$S_\pi \in \mathcal{S}_N$ is symmetric if, and only if the corresponding permutation $\pi \in \Sigma_N$ 
is an involution (a.k.a. self inverse), i.e., if $\pi(\pi(i))=i$ holds for every $i \in \{1,\ldots,N\}$. 
In accordance with shuffles in the sequel we will simply refer to self inverse permutations as symmetric.  
The subclass of all symmetric elements of $\mathcal{S}_N$ will be denoted by $\mathcal{S}_N^{sym}$. Furthermore, 
to simplify notation we will write 
$$
\mathcal{S} := \bigcup_{N\in\N} \mathcal{S}_N, \quad \mathcal{S}^{sym} := \bigcup_{N\in\N} \mathcal{S}_N^{sym}
$$
and refer to elements of $\mathcal{S}$ (or $\mathcal{S}^{sym} $) as shuffles (or symmetric shuffles). 
As commonly done in the literature we will also refer to the corresponding mutually completely
dependent copula $C_h$ as shuffle (or symmetric shuffle) and write 
\begin{equation}
\mathcal{C}_\mathcal{S} := \left\{C_h: \, h\in \mathcal{S} \right\}, \quad 
\mathcal{C}_{\mathcal{S}^{sym}} := \left\{C_h: \, h\in \mathcal{S}^{sym} \right\}.
\end{equation}

According to \cite{edwards2004measures} a mapping $\kappa: \C \to \R$ is called a measure of concordance if it satisfies the following properties:
\begin{enumerate}
    \item[(i)] $\kappa(M) = 1$;
    \item[(ii)] $\kappa(C^t) = \kappa(C)$ for all $C\in\C$;
    \item[(iii)] $\kappa(C^\nu) = -\kappa(C)$ for all $C\in\C$ where $C^\nu$ 
    is the reflection of $C$ at $u=\frac{1}{2}$, i.e., $C^\nu(u,v) := v - C(1-u,v) $
    \item[(iv)] $\kappa(C) \leq \kappa(D)$ whenever $C\leq D$,
    i.e., whenever $C$ and $D$ are ordered pointwise;
    \item[(v)] $\lim_{n \to \infty} \kappa (C_n) = \kappa(C)$ for any sequence $(C_n)_{n \in \N}$ of 
    copulas converging to $C\in\C$.
\end{enumerate}

It is well-known (see \cite{nelsen2007introduction}) that Spearman's $\rho$ and Spearman's footrule $\phi$ 
can be expressed in terms of the underlying copula as follows:
\begin{align}
    \rho(C) &= 12 \int_{[0,1]^2} C(u,v) d\mu_{\Pi}(u,v) - 3 =  12 \int_{[0,1]^2} C(u,v) d\lambda_2(u,v) - 3 \label{DefRho} \\
    \phi(C) &= 6 \int_{[0,1]^2} C(u,v) d\mu_{M}(u,v) - 2 = 6 \int_{[0,1]} C(u,u) d\lambda(u) - 2 \label{DefPhi}
\end{align}
Considering $\phi(M^\nu)=\phi(W) = - \frac{1}{2} \neq -1 = - \Phi(M)$ 
Spearman's footrule $\phi$ is only a weak measure of concordance. To simplify notation, for $S_\pi \in \mathcal{S}_N$ 
we will also write $\rho(S_\pi):=\rho(C_{S_\pi})$ as well as $\phi(S_\pi):=\phi(C_{S_\pi})$ in the sequel. \\

Finally, denoting the diagonal of a copula $C$ by $\delta_C$, i.e., $\delta_C(t) := C(t,t)$ for every $t\in[0,1]$, 
it is well known (see \cite{sanchez2016some}) that $\delta_C$ satisfies the following properties:
\begin{itemize}
    \item $\delta_C(0) = 0$ and $\delta_C(1) = 1$,
    \item $\delta_C$ is monotonically non-decreasing,
    \item $\delta_C$ is Lipschitz continuous with Lipschitz constant $L = 2$ and
    \item $\delta_C(t) \leq t$ for all $t\in[0,1]$.
\end{itemize}
In the sequel $\D$ denotes the family of all diagonals of copulas (which is well known to coincide 
with the class of all functions
$\delta: [0,1] \rightarrow [0,1]$ fulfilling the afore-mentioned four points). 
For every $\delta \in \D$, setting $\Hat{\delta}(t) := t - \delta(t)$ for all $t\in[0,1]$ it 
follows that both $\delta$ and $\Hat{\delta}$ are differentiable $\lambda$-almost everywhere (see \cite{rudin1987real}), 
hence there exist some measurable functions $w_\delta: [0,1] \to [0,2]$ and $\Hat{w}_\delta: [0,1] \to [-1,1]$ 
with $w_\delta(x) = \delta'(x) $ as well as $\Hat{w}_\delta(t) = 1 - \Hat{\delta}'(t) $ for $\lambda$-almost every $t\in[0,1]$.
We will refer to $w_\delta$ and $\Hat{w}_\delta$ as measurable versions of the derivative of $\delta$ and $\hat{\delta}$, 
respectively.
\clearpage

\section{Novel proofs for the upper bound}\label{SecUpperBound}

\noindent We first tackle the upper bound 
\begin{align}\label{UpperBound}
    \rho(C) \leq 1 - \frac{2}{3} \left(\phi(C) - 1 \right)^2
\end{align}
going back to \cite{bukovvsek2024exact} and established by working with diagonal copulas, sufficiently smooth diagonals 
and Bernstein approximations (the technically quite involved Lemma 9 being key). 
We provide two simple alternative proofs, one purely based on Cauchy-Schwarz inequality, and
the other one using maximality properties of diagonal copulas.  

\subsection{A simple proof via symmetric shuffles and Cauchy-Schwarz inequality}
Building upon the fact that for every copula $C$ the symmetric copula $C^\ast := \frac{1}{2} (C + C^t)$ fulfills $\rho(C^\ast) = \rho(C)$ and $\phi(C^\ast) = \phi(C)$, it suffices to prove ineq. \eqref{UpperBound} for the class 
$\Csym$ of all symmetric copulas. Moreover, using continuity of $\rho$ and $\phi$ w.r.t $d_\infty$ 
we can further reduce the problem to any dense subclass of the family $\Csym$. 
The following lemma will therefore be key. 

\begin{Lemma}\label{DenseSymShuffles}
    The family of all symmetric shuffles $\mathcal{C}_{\mathcal{S}^{sym}}$ is dense in $(\Csym, d_\infty)$. 
\end{Lemma}
\begin{proof}
Looking into the proof of the corresponding result for the full class $(\C, d_\infty)$ 
in \cite{mikusinski1992shuffles} reveals the fact
that, starting with an arbitrary symmetric copula $C$ the constructed approximating shuffle is symmetric 
as well. In other words: the original proof directly yields Lemma \ref{DenseSymShuffles}.
\end{proof}

Next we derive handy formulas for Spearman's $\rho$ and Spearman's footrule $\phi$ for 
$C_{S_\pi}$ with $S_\pi \in \mathcal{S}_N^{sym}$ and work with the following sets:  
\begin{eqnarray}
  I_\pi^- &:=& \{i \in \{1, \dots, N\}: \pi(i) < i \}, \quad I_\pi^0 := \{i \in \{1, \dots, N\}: \pi(i) = i \}, \nonumber \\
     I_\pi^+ &:=& \{i \in \{1, \dots, N\}: \pi(i) > i \}.
\end{eqnarray}   
We will only write $I^-, I^0$ and $I^+$ whenever no confusion can arise.
Notice that symmetry of $S_\pi \in \mathcal{S}_N^{sym}$ implies that 
(i) $i \in I^- $ if, and only if $\pi(i) \in I^+$ and that (ii) $i \in I^+ $ if, and only if $\pi(i) \in I^-$.

\begin{Lemma}\label{FormulasRhoPhi}
    For every $N \in \N$ and $S_\pi \in \mathcal{S}_N^{sym}$ the following identities hold:
    \begin{align}
        \rho(S_\pi) &= 1 - \frac{12}{N} \sum_{i \in I^-} \left( \frac{i - \pi(i)}{N} \right)^2, \label{Rho} \\
        \phi(S_\pi) &= 1 - \frac{6}{N} \sum_{i \in I^-} \frac{i - \pi(i)}{N}. \label{Phi}
    \end{align}
\end{Lemma}

\begin{proof}
    Fix  $N \in \N$ and $S_\pi \in \mathcal{S}_N$. Then obviously (a version of) the Markov kernel 
    is given by $K_{S_\pi}(x,F) = \mathds{1}_F(S_\pi(x))$ with 
    $$
    S_\pi(x) = \sum_{i=1}^N \underbrace{\left(x - \frac{i - \pi(i)}{N}\right)}_{=:h_i(x)}\mathds{1}_{I^i_N}(x).
    $$  
    Using disintegration therefore yields (for a justification for the interchange of the order of integration see, e.g., 
    \cite{edwards2004measures}) 
       \begin{align*}
        \frac{\phi(S_\pi) + 2}{6} &= \int_{[0,1]^2} S_\pi(u,v) d\mu_M(u,v) = 
        \int_{[0,1]^2} M(u,v) d\mu_{S_\pi}(u,v) \\&= \int_{[0,1]} M(u,S_\pi(u)) d\lambda(u) = 
        \sum_{i=1}^N \int_{I^i_N} M(u,h_i(u)) d\lambda(u) \\&=
        \sum_{i: \pi(i) \geq i}  \int_{I^i_N} u d\lambda(u) + \sum_{i: \pi(i) < i}  \int_{I^i_N} u - \frac{i - \pi(i)}{N} d\lambda(u) \\&=
        \sum_{i=1}^N \int_{I^i_N} u d\lambda(u) - \frac{1}{N} \sum_{i: \pi(i) < i} \frac{i - \pi(i)}{N} \\&=
        %\int_{[0,1]} u d\lambda(u) - \frac{1}{N} \sum_{i \in I^-} \frac{i - \pi(i)}{N} =
        \frac{1}{2} - \frac{1}{N} \sum_{i \in I^-} \frac{i - \pi(i)}{N}.
    \end{align*}
    Proceeding analogously for Spearman's $\rho$ we get 
    \begin{align*}
        \frac{\rho(S_\pi) + 3}{12} &= \int_{[0,1]^2} S_\pi(u,v) d\mu_\Pi(u,v) = \int_{[0,1]^2} \Pi(u,v) d\mu_{S_\pi}(u,v) \\&= \int_{[0,1]} \Pi(u,S_\pi(u)) d\lambda(u) = 
        \sum_{i=1}^N \int_{I^i_N} \Pi(u,h_i(u)) d\lambda(u) \\&= \sum_{i=1}^N \int_{I^i_N} u^2 - u \,\frac{i - \pi(i)}{N} d\lambda(u) \\&=
        \int_{[0,1]} u^2 d \lambda(u) - \sum_{i=1}^N  \frac{i - \pi(i)}{N} \frac{i^2 - (i-1)^2}{2N^2} 
        \\&= \frac{1}{3} - \sum_{i=1}^N \frac{i - \pi(i)}{N} \frac{2i-1}{2N^2}.
    \end{align*}
    Since $\pi \in \Sigma_N$ we have $ \sum_{i=1}^N (i - \pi(i)) = 0$, which altogether yields
       \begin{align*}
        \rho(S_\pi) = 1 -6 \sum_{i=1}^N \frac{i - \pi(i)}{N} \frac{2i-1}{N^2} = 1 - \frac{12}{N} \sum_{i=1}^N \frac{i(i -\pi(i))}{N^2}.
    \end{align*}   
    Finally, using symmetry of $S_\pi$ the last summand simplifies to 
    \begin{align*}
        \sum_{i=1}^N i(i - \pi(i)) &= \sum_{i \in I^-} i(i - \pi(i)) + \sum_{i \in I^0} i(i - \pi(i)) + \sum_{i \in I^+} i(i - \pi(i)) \\&=
        \sum_{i \in I^-} i(i - \pi(i)) + \sum_{j \in I^-} \pi(j)(\pi(j) - \pi(\pi(j))) \\&=
        \sum_{i \in I^-} i(i - \pi(i)) + \pi(i)(\pi(i) - i) =\sum_{i \in I^-} (i - \pi(i))^2,
    \end{align*}   
    which completes the proof.
\end{proof}

According to Lemma \ref{FormulasRhoPhi}, the values of the permutation $\pi$ on the set $I^-$ contain 
all relevant information for calculating $\rho$ and $\phi$. 
This simple observation opens the door for applying Cauchy-Schwarz inequality and deriving a very simple proof 
for the upper inequality ($\#I^- $ denoting the cardinality of $I^-$): 
\begin{Theorem}\label{UpperboundShuffle}
    For every $N\in\N$  and every symmetric shuffle $S_\pi \in \mathcal{S}_N^{sym}$ the inequality 
    \begin{align}\label{UpperboundShuffleEq}
    \rho(S_N) \leq 1 - \frac{2}{3} \left(1 - \phi(S_N)\right)^2
    \end{align}
    holds. Furthermore we have equality in \eqref{UpperboundShuffleEq} if, and only if $i \mapsto \pi(i)-i$ 
    is constant on $I^-$ and $\# I^- = \frac{N}{2}$ holds.
\end{Theorem}

\begin{proof}
Since the result is obvious for $I^-=\emptyset$ it suffices to consider $I^-\neq \emptyset$.
Cauchy-Schwarz inequality implies 
    \begin{align}\label{FormulasRhoPhiProof}
        \sum_{i\in I^-} \frac{i - \pi(i)}{N} \leq (\#I^-)^{1/2} \left(  \sum_{i \in I^-} \left( \frac{i - \pi(i)}{N} \right)^2 \right)^{1/2},
    \end{align}
    so using eqs. \eqref{Rho} and \eqref{Phi} shows
    \begin{align}\label{eq:temp}
    \left( \frac{1 - \phi(S_N)}{6} N \right)^2 \leq (\#I^-)^{1/2}  \left( \frac{1-\rho(S_N)}{12} N \right) 
    \leq \frac{N}{2} \frac{1-\rho(S_N)}{12}\,N,
    \end{align}
    which directly yields the desired inequality $\rho(S_N) \leq 1 - \frac{2}{3} \left(1 - \phi(S_N) \right)^2$.  \\
    Turning toward sharpness, obviously ineq. \eqref{FormulasRhoPhiProof} is sharp if, and only if the mapping
    $i \mapsto\frac{i - \pi(i)}{N}$ is constant on $I^-$. Furthermore the second part of ineq. \eqref{eq:temp} is 
    sharp if, and only if $\#I^- = \frac{N}{2}$. 
    \end{proof}

\noindent For an even $n \in\N$ let $S_{\pi^\ast} \in \mathcal{S}_n^{sym}$ denote the shuffle corresponding to  
$\pi^\ast \in \Sigma_n$ fulfilling 
$i - \pi^\ast(i) = 1$ for every $i \in I^- =  \{i \in \{1, \dots, n\}: i \text{ even}\}$; for the case $n=6$ see 
Figure \ref{FigShuffleUpperBound}. 
Then obviously the mapping $i \mapsto\frac{i - \pi^\ast(i)}{N}$ is constant on $I^-$, so 
we already know from Theorem \ref{UpperboundShuffleEq} that ineq. (\ref{UpperboundShuffleEq}) becomes
an equality (compare with Example 12 in \cite{bukovvsek2024exact}). 

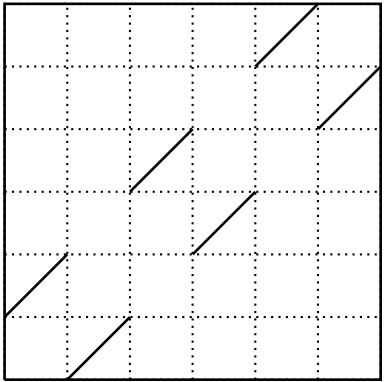
\begin{figure}[H]
    \centering
	\begin{tikzpicture}[scale=5]
	\draw[-,line width=1] (0,0) -- (0,1) -- (1,1) -- (1,0) -- (0,0);
        % Support 
        \draw[-,line width=1] (0,1/6) -- (1/6,2/6);
        \draw[-,line width=1] (1/6,0) -- (2/6,1/6);
        \draw[-,line width=1] (2/6,3/6) -- (3/6,4/6);
        \draw[-,line width=1] (3/6,2/6) -- (4/6,3/6);
        \draw[-,line width=1] (4/6,5/6) -- (5/6,6/6);
        \draw[-,line width=1] (5/6,4/6) -- (6/6,5/6);

        % Grid
        \draw[step=1/6,black,thick,xshift=0cm,yshift=0cm,dotted] (0,0) grid (1,1);
        
	\end{tikzpicture}
	\caption{The shuffle $S_{\pi^\ast} \in \mathcal{S}_6^{sym}$ for which ineq. (\ref{UpperboundShuffleEq}) 
	becomes an equality.} \label{FigShuffleUpperBound}
\end{figure}

Combining Theorem \ref{UpperboundShuffle} and Lemma \ref{DenseSymShuffles} already yields the following 
inequality for all copulas: 
\begin{Theorem}[\cite{bukovvsek2024exact}]\label{ThmUpperBound}
    For every copula $C\in\C$ the following inequality holds:
    \begin{align*}
        \rho(C) \leq 1 - \frac{2}{3} \left( 1 - \phi(C) \right)^2.
    \end{align*}
\end{Theorem}

\subsection{A second alternative proof via diagonal copulas and their interrelation with shuffles}

It is well-known (see \cite{sanchez2016some, ubeda2008best, nelsen2004best}) that, given a diagonal $\delta \in \mathcal{D}$, 
the diagonal copula $E_\delta$, given by 
    \begin{align*}
        E_\delta(u,v) = \min \Big\{ u, v, \frac{\delta(u) + \delta(v)}{2} \Big\}
    \end{align*}
for all $u,v \in [0,1]$ is the maximal element in the class of all copulas with diagonal $\delta$, i.e., 
$C \leq E_\delta$ for all copulas $C\in\Csym$ with diagonal $\delta\in\D$. 

Before proceeding with the alternative proof of Theorem \ref{ThmUpperBound} we recall some properties 
of diagonal copulas going back to \cite{sanchez2016some}), which will prove useful in the sequel.
Obviously the mapping $\iota: (\mathcal{D},\Vert \cdot \Vert_\infty) \rightarrow (\mathcal{C},d_\infty)$, defined by 
$\iota(\delta)=E_\delta$ is continuous. Moreover, for a given diagonal $\delta \in\D$ set $g(t) := 2t - \delta(t)$ and define
\begin{align*}
    L(t) := \min \{ z\in[0,1]: g(z) \geq \delta(t) \} \text{ and } U(t) := \min \{ z\in[0,1]: \delta(z) \geq g(t) \}.
\end{align*}
It is straightforward to verify that both $L$ and $U$ are non-decreasing and that 
$L(t) \leq t \leq U(t)$ holds for every $t\in[0,1]$. 
Furthermore, it can be shown that the diagonal copula $E_\delta$ distributes its mass on the graphs of the 
functions $L$ and $U$. More precisely, the following result holds:
\begin{Proposition}[\cite{sanchez2016some}]\label{MarkovKernelDiagonalCopula}
   Suppose that $\delta \in \mathcal{D}$ and let $w_\delta$ denote a measurable version of its derivative. 
   Then (a version of) the Markov kernel $K_{E_\delta}$ of $E_\delta$ is given by  
    \begin{align*}
        K_{E_\delta}(t, F) = \frac{w_\delta(t)}{2} \mathds{1}_{L(t)}(F) + 
        \left( 1 - \frac{w_\delta(t)}{2} \right) \mathds{1}_{U(t)}(F).
    \end{align*}
  Moreover,  $E_\delta$ is (mutually) completely dependent and concentrates its mass on the graph of a 
  $\lambda$-preserving bijection $h: [0,1] \to [0,1]$ fulfilling $h \circ h = id_{[0,1]}$ if, and only if for 
  $\lambda$-almost every $x\in[0,1]$ either $\delta'(x) \in \{0, 2\}$ or $\delta(x) = x$ holds. 
  \end{Proposition}
 For every $N \in \mathbb{N}$ we will let $\mathcal{D}_N^{0,2}$ denote the family of 
 all diagonals $\delta \in \mathcal{D}$ fulfilling that on each open interval $I_N^i$ we either have 
 that (i) $\delta'(x)=0$ for all $x \in I_N^i$ or that (ii) $\delta'(x)=2$ for all $x \in I_N^i$. 
 Figure \ref{ExampleDiagDense} depicts an example of such a diagonal. Obviously 
 $\mathcal{D}_N^{0,2} = \emptyset$ for odd $N \in \mathbb{N}$ and 
 $\mathcal{D}_N^{0,2} \neq \emptyset$ for even $N \in \mathbb{N}$. To simplify notation set
 $\D_\infty^{0,2} := \bigcup_{N \in \N} \D_N^{0,2}$. \\
 
Using this notation Proposition \ref{MarkovKernelDiagonalCopula} opens the door to the following alternative 
idea of proof for ineq. (\ref{UpperBound}), which we will now tackle step by step: \\
(Step 1) Show that the family $\mathcal{D}_\infty^{0,2}$ is dense in $(\D, \Vert . \Vert_\infty)$. \\
(Step 2) Show that for every $\delta \in \mathcal{D}_N^{0,2}$ the diagonal copula $E_\delta$ is a symmetric, equidistant even 
shuffle, i.e., $E_\delta \in \mathcal{C}_{\mathcal{S}^{sym}}$.\\
(Step 3) Use the maximality property of diagonal copulas (mentioned before) 
and apply Theorem \ref{UpperboundShuffle}. 
\begin{Lemma}\label{DenseDiagonals}
    $\D_\infty^{0,2}$ is dense in $(\D, \Vert . \Vert_\infty)$.  
\end{Lemma}

\begin{proof}
    Let $\delta \in \D$ and $N \in \N$ be arbitrary but fixed. We construct an element 
    $\Tilde{\delta} \in \D_{2N}^{0,2}$ as follows: For every $i \in \{1,\ldots,2N\}$ consider   
     $y_i := \delta(\frac{i}{2N}) $ and define 
     $$
     i_k := \min \left\{ i \in \{1, \dots , 2N\}: y_i \geq \frac{2k}{2N}=\frac{k}{N} \right\}
     $$ 
     for every $k \in \{1,\ldots,N\}$. Setting $i_0:=0$ this obviously yields 
     $0 = i_0 < 2 \leq i_1 < i_2 < \dots < i_{N-1} < i_N = 2N$. 
     Defining    
    \begin{align*}
        f(x) := 2 \sum_{k=1}^N \mathds{1}_{(\frac{i_k - 1}{2N}, \frac{i_k}{2N})}(x).
    \end{align*}
    it follows immediately that $f$ is a probability density (w.r.t. $\lambda$), so the 
    function $\Tilde{\delta}:[0,1] \rightarrow [0,1]$, defined by
    $$
    \Tilde{\delta}(x) = \int_{[0,x]} f \,d\lambda
    $$ 
    obviously is a Lipschitz continuous (Lipschitz constant $L=2$) piecewise linear distribution function 
    fulfilling $\Tilde{\delta}(0)=0$ and  $\Tilde{\delta}(1)=1$. Furthermore $\Tilde{\delta} \geq \delta_W$ 
    (otherwise $\Tilde{\delta}(1) = 1$ is impossible) and by construction of $f$ we have that
    \begin{align*}
        \Tilde{\delta}\left(\frac{i_k}{2N} \right) &= \frac{2k}{2N} \leq y_{i_k}= \delta\left(\frac{i_k}{2N} \right)
    \end{align*}
    holds for every $k$, implying $\Tilde{\delta} \leq \delta$ on 
    $[0,1]$ since $f$ is $0$ outside $\cup_{k=1}^N (\frac{i_k - 1}{2N}, \frac{i_k}{2N})$. Altogether 
    $\Tilde{\delta} \in \D_{2N}^{0,2}$ follows. \\
    Finally, considering that on the grid $\{0,\frac{i_1}{2N},\frac{i_2}{2N},\ldots,\frac{i_{N-1}}{2N},1\}$ 
    the two diagonals have a maximum distance of at most $\frac{2}{2N}=\frac{1}{N}$, using monotonicity 
    and the fact that on each interval $I_{2N}^i$ the diagonal $\Tilde{\delta}$ has either slope $0$ or $2$
    it follows that   
    $$
      \Vert \delta - \Tilde{\delta} \Vert_\infty \leq \frac{1}{N}.
    $$
    This completes the proof since for sufficiently large $N$ the quantity $\frac{1}{N}$ is smaller than any fixed 
    $\varepsilon>0$. 
 \end{proof}
 We will now clarify under which conditions diagonal copulas are equidistant even shuffles (and vice versa). 
 Doing so we will call $\pi \in \Sigma_N$ bi-monotone if $\pi$ restricted to $I_\pi^-$ is strictly increasing.
 Notice that if $\pi$ is symmetric and bi-monotone then $\pi$ restricted to $I_\pi^+$ is strictly increasing too.
 The right panel of Figure \ref{ExampleDiagDense} depicts a shuffle $S_\pi$ corresponding
 to a bi-monotone symmetric permutation $\pi \in \Sigma_{12}$. 
  
 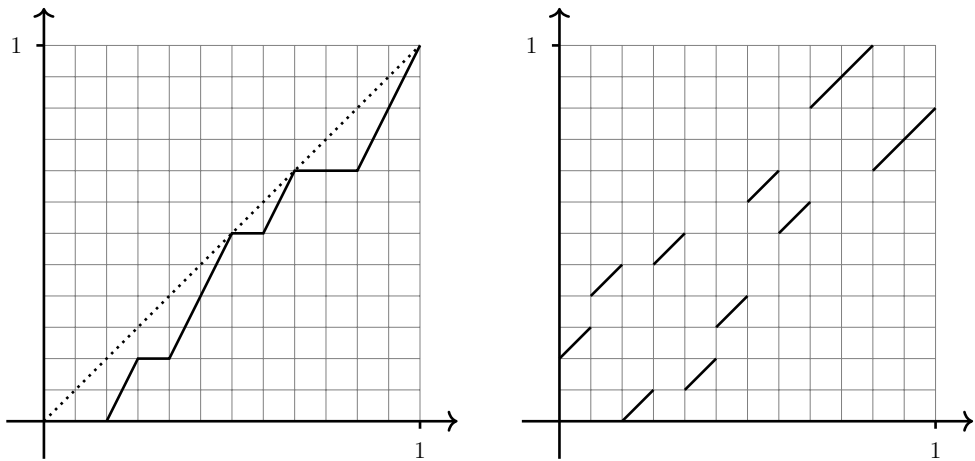
\begin{figure}[H]
    \centering
    \begin{subfigure}{.5\textwidth}
        \centering
        \begin{tikzpicture}[scale=5]
    	\draw[step=1/12, gray, very thin] (0,0) grid (1,1); 
    	\draw[->,line width=1] (-0.1,0) -- (1.1,0);
            \draw[->,line width=1] (0,-0.1) -- (0,1.1);
            \draw[-,line width=1,dotted] (0,0) -- (1,1);
            
            \draw[-,line width=1] (2/12,0) -- (3/12,2/12) -- (4/12,2/12) -- (6/12,6/12) -- (7/12,6/12) -- (8/12,8/12) -- (10/12,8/12) -- (1,1); 
    
            % x-Axis Label
            \draw[-,line width=1] (1,0) -- (1,-0.02);
    	    \node[below=1pt of {(1,-0.02)}, scale= 0.75, outer sep=2pt,fill=white] {$1$};
            % y-Axis Label
            \draw[-,line width=1] (0,1) -- (-0.02,1);
    	    \node[left=1pt of {(-0.02,1)}, scale= 0.75, outer sep=2pt,fill=white] {$1$};
        \end{tikzpicture}
    \end{subfigure}%
    %\hspace{0.25cm}
    \begin{subfigure}{.5\textwidth}
        \centering
        \begin{tikzpicture}[scale=5]
    	\draw[step=1/12, gray, very thin] (0,0) grid (1,1); 
    	\draw[->,line width=1] (-0.1,0) -- (1.1,0);
            \draw[->,line width=1] (0,-0.1) -- (0,1.1);
            %\draw[-,line width=1,dotted] (0,0) -- (1,1);
            
            %\draw[-,line width=1] (2/12,0) -- (3/12,2/12) -- (4/12,2/12) -- (6/12,6/12) -- (7/12,6/12) -- (8/12,8/12) -- (10/12,8/12) -- (1,1); 
    
            % x-Axis Label
            \draw[-,line width=1] (1,0) -- (1,-0.02);
    	    \node[below=1pt of {(1,-0.02)}, scale= 0.75, outer sep=2pt,fill=white] {$1$};
            % y-Axis Label
            \draw[-,line width=1] (0,1) -- (-0.02,1);
    	    \node[left=1pt of {(-0.02,1)}, scale= 0.75, outer sep=2pt,fill=white] {$1$};

            % support L
            \draw[-,line width=1] (2/12,0/12) -- (3/12,1/12);
            \draw[-,line width=1] (4/12,1/12) -- (5/12,2/12);
            \draw[-,line width=1] (5/12,3/12) -- (6/12,4/12);
            \draw[-,line width=1] (7/12,6/12) -- (8/12,7/12);
            \draw[-,line width=1] (10/12,8/12) -- (12/12,10/12);

            % support U
            \draw[-,line width=1] (0/12,2/12) -- (1/12,3/12);
            \draw[-,line width=1] (1/12,4/12) -- (2/12,5/12);
            \draw[-,line width=1] (3/12,5/12) -- (4/12,6/12);
            \draw[-,line width=1] (6/12,7/12) -- (7/12,8/12);
            \draw[-,line width=1] (8/12,10/12) -- (10/12,12/12);
            
        \end{tikzpicture}
    \end{subfigure}
    \caption{Example of a diagonal $\delta \in \D_{12}^{0,2}$ (left panel) and the corresponding diagonal co\-pula/shuffle
    $E_\delta=C_{S_\pi}$ with $\pi=(3,5,1,6,2,4,8,7,11,12,9,10)$ and $I_\pi^-=(3,5,6,8,11,12)$ (right panel).}
         \label{ExampleDiagDense}
\end{figure}

\begin{Theorem}\label{Delta_N_EquidistantShuffle}
    Suppose that $N \in \N$ is even and that $\delta \in \D_N^{0,2}$. Then the diagonal copula $E_\delta$ 
    is an equidistant even shuffle, i.e., 
    $E_\delta \in \mathcal{C}_{\mathcal{S}_N^{sym}} \subseteq \mathcal{C}_{\mathcal{S}^{sym}}$ and 
    the corresponding permutation $\pi\in\Sigma_N$ is symmetric, bi-monotone and fulfills $I_\pi^1=\emptyset$. 
\end{Theorem}

\begin{proof}
    We already know from Proposition \ref{MarkovKernelDiagonalCopula} that under the assumptions of the theorem
    the corresponding diagonal copula $E_\delta$ is mutually completely dependent and that there exists some
    $\lambda$-preserving, bijective $h:[0,1] \rightarrow [0,1]$ fulfilling $h \circ h=id$ 
    such that $E_\delta=C_h$. It remains to show that $h \in \mathcal{S}_N^{sym}$ and that the corresponding
    permutation $\pi$ is bi-monotone (the fact that $\pi \circ \pi$ is a direct consequence of symmetry of $E_\delta$).  
    Defining
     \begin{eqnarray*}
        J^0 &:=& \{i \in \{1, \dots, N\}: \delta'(x) = 0 \text{ for every } x \in I_N^i\}, \\
        J^2 &:=& \{i \in \{1, \dots, N\}: \delta'(x) = 2 \text{ for every } x \in I_N^i\}.
    \end{eqnarray*}
   we have $J^0 \cup J^2 =  \{1, \dots, N\}$.   
   Notice that on every interval $I_N^j$ with $j \in J^0$ the function $L$ is constant, whereas 
   $U$ is constant on every interval $I_N^j$ with $j \in J^2$. Furthermore interpreting 
   $\delta$ and $g$ as distribution functions with quasi-inverses $\delta^-$ and $g^-$, respectively,
   it follows that $L=g^- \circ \delta$ and that  
   $$
   g^-([0,1])=  \bigcup_{j \in J^0} \overline{I}_N^j
   $$
   holds, whereby $\overline{I}_N^i$ denotes the closure of the open interval $I_N^i$. As a direct consequence, 
   for every $x \in I_N^j$ with $j \in J^2$ we have that $L(x) \in \bigcup_{j \in J^0} I_N^j$. 
   Proceeding analogously for $U$ shows that for every $x \in I_N^j$ with $j \in J^0$ we have that 
   $U(x) \in \bigcup_{j \in J^2} I_N^j$. Finally considering at every point of differentiability 
   $L$ and $U$ can only have slope $0$ or $1$ by the chain rule it follows that $E_\delta$ 
   is indeed an equidistant even shuffle. 
   Letting $\pi \in \Sigma_N$ denote the corresponding permutation we get that 
   $\pi$ maps $J^0$ (bijectively) to $J^2$ and vice versa and that $I_\pi^-=J^2$ holds. 
   Furthermore, using the fact that $L,U$ are non-decreasing,   
   $\pi$ is strictly increasing on $J^0$ and on $J^2$, i.e., $\pi$ is bi-monotone. \\
   Finally, considering that according to \cite{sanchez2016some} $\delta(t) < t $ implies $ L(t) < t $ and $ U(t) > t$
   and that for $\delta \in \D_N^{0,2}$ obviously $\delta(t) < t$ holds for all but at most $N$ points, 
   we conclude that $I_\pi^1=\emptyset$ and the proof is complete.
\end{proof}

Now Step 3 is obvious - combining Lemma \ref{DenseDiagonals} and Theorem \ref{Delta_N_EquidistantShuffle}
directly completes our second alternative proof for the upper inequality.  \\

\noindent Considering that (to the best of our knowledge) 
the interplay between diagonal copulas and equidistant even shuffles hasn't be studied yet we 
conclude this section with the converse of Theorem \ref{Delta_N_EquidistantShuffle}.
\begin{Proposition}
    Let $N\in\N$ be even, $\pi \in \Sigma_N$ be a symmetric, bi-monotone permutation with $I_\pi^1=\emptyset$, and $S_\pi$ 
    denote the corresponding shuffle. Then the shuffle $C_{S_\pi}$ is a diagonal copula.
\end{Proposition}

\begin{proof}
If $\pi \in \Sigma_N$ fulfills the assumptions of the proposition, then letting $S_\pi$ denote the 
corresponding shuffle we have 
$\delta:=\delta_{C_{S_\pi}} \in \mathcal{D}_N^{0,2}$, so according to Theorem \ref{Delta_N_EquidistantShuffle} 
the induced diagonal copula $E_\delta$ is an equidistant shuffle and the corresponding 
permutation $\pi^\ast$ is symmetric, bi-monotone and fulfills $I_{\pi^\ast}^1=\emptyset$. 
Furthermore (as shown in the proof of Theorem \ref{Delta_N_EquidistantShuffle}) we have $I_{\pi^\ast}^-=J^2=I_{\pi}^-$.
Considering the facts that $\pi$ is strictly increasing on $I_{\pi}^-$, that $\pi^\ast$ is strictly increasing 
on $I_{\pi^\ast}^-$, and that the two sets coincide, the identity $\pi = \pi^\ast$ follows immediately, and the 
proof is complete.  
\end{proof}

\clearpage
\section{A novel proof for the sharp lower bound}\label{SecLowerBound}
\noindent In this section we focus on the inequality 
\begin{align}\label{LowerBound}
    \rho(C) \geq \frac{2}{9} \sqrt{3} \left( 1 + 2\phi(C) \right)^{3/2} - 1.
\end{align}
In the original paper \cite{bukovvsek2024exact} the authors derived this very equality 
by first showing it for the subclass of copulas assigning full mass to the main and second diagonal (i.e., 
$\mathbb{P}(X=Y) + \mathbb{P}(Y=1-X)=1$ with $(X,Y)\sim C$) and then extending the results to the full class. 
The chosen method of proof is interesting and novel but at the same time technically quite involved. 
In the sequel we show that working with shuffles, another (seemingly novel) mass rearrangement idea, and 
denseness arguments allows for a shorter and less technical alternative proof.   \\

We start with some first observations on the lower inequality, motivate the mass rearrangement idea, then 
prove a much more general mass rearrangement result for $L_2$ functions, and finally apply it to 
derive ineq. (\ref{LowerBound}). \\
As in the previous section we will work with symmetric permutations $\pi \in \Sigma_N$, the shuffles $S_\pi$ and 
the corresponding sets $I_\pi^-$. To simplify notation throughout this section we will write $k:=\#I_\pi^- \leq \frac{N}{2}$.
For $k=0$ we obviously have $S_\pi=id_{[0,1]}$, implying $C_{S_\pi}=M$, which yields equality in (\ref{LowerBound}).
It therefore suffices to consider $k\geq 1$. \\
For $k\geq 1$ we will let $0 < p_1 \leq \dots \leq p_k$ denote the order statistics of the points 
$\frac{i_l-\pi(i_l)}{N}$ with $l \in \{1,\ldots,k\}$ and $I_\pi^-=\{i_1,\ldots,i_k\} \subset \{1,\ldots,N\}$, and will 
write $\textbf{p} = (p_1, \dots, p_k) \in [0,1]^k$. 
Using this notation the formulas for Spearman's $\rho$ and Spearman's footrule $\phi$ 
(see Lemma \ref{FormulasRhoPhi}) simplify to
    \begin{align*}
        \rho(S_\pi) &= 1 - \frac{12}{N} \sum_{i = 1}^{k} p_i^2, \quad \phi(S_\pi) = 1 - \frac{6}{N} \sum_{i = 1}^{k} p_i.
    \end{align*}
    Hence the desired inequality is equivalent to
    \begin{align*}
          \frac{2}{9} \sqrt{3} \left( 3 - \frac{12}{N} \sum_{i = 1}^k p_i \right)^{3/2} - 1 \leq 1 - \frac{12}{N}
           \sum_{i = 1}^k p_i^2
     \end{align*}
     which, in turn, simplifies to 
     \begin{align*}       
       \left( 1 - \frac{4}{N} \sum_{i = 1}^k p_i \right)^{3/2} \leq 1 - \frac{6}{N} \sum_{i = 1}^k p_i^2.
    \end{align*}
    In other words, ineq. (\ref{LowerBound}) boils down to showing that 
    \begin{align}\label{def:f.lower}
     m_{\pi,k}(\textbf{p}) := 1 - \frac{6}{N} \sum_{i = 1}^k p_i^2 - \left( 1 - \frac{4}{N} 
     \sum_{i = 1}^k p_i \right)^{3/2}  \geq 0
    \end{align}
    for all symmetric permutations $\pi$.  
    
    The following example illustrates the idea underlying the rearrangement, which we will 
    work with in this section:   
\begin{Example}\label{ex:rearrange}
\emph{
Consider $N=8$ and the symmetric permutation $\pi \in \Sigma_8$, given by $\pi=(4,7,8,1,6,5,2,3)$. 
Then we have $I_\pi^-=\{4,6,7,8\}, k=4$, and $\frac{i_1-\pi(i_1)}{N} = \frac{3}{8}, \frac{i_2-\pi(i_2)}{N} = \frac{1}{8},
\frac{i_3-\pi(i_3)}{N} = \frac{5}{8}, \frac{i_4-\pi(i_4)}{N} = \frac{5}{8}$, so the vector $\textbf{p}$ is given by 
$\mathbf{p}=\frac{1}{8}\cdot (1,3,5,5)$. The left panel of Figure \ref{MassReDistributionEx} depicts
the shuffle $C_{S_\pi}$, the shaded squares illustrate $\mathbf{p}$ and $I_\pi^-$, their total number is 
$\Delta=N \cdot \sum_{i=1}^kp_i=14$. Notice that $S_{\hat{\pi}}$ in the right panel can be constructed from $S_\pi$ 
by shifting the squares as much to the right as possible., i.e., the symmetric permutation $\hat{\pi}$ 
is given by 
$\hat{\pi}=(8,7,3,6,5,4,2,1) \in \Sigma_8$. Denoting all quantities corresponding $\hat{\pi}$ with 
a hat, for $\hat{\pi}$ we have $I_{\hat{\pi}}^-=\{6,7,8\}, \hat{k}=3$, and
$\hat{\mathbf{p}}=\frac{1}{8}\cdot(2,5,7)$, implying $\hat{\Delta}=N \cdot \sum_{i=1}^{\hat{k}}\hat{p}_i=14=\Delta$.
\\
}
    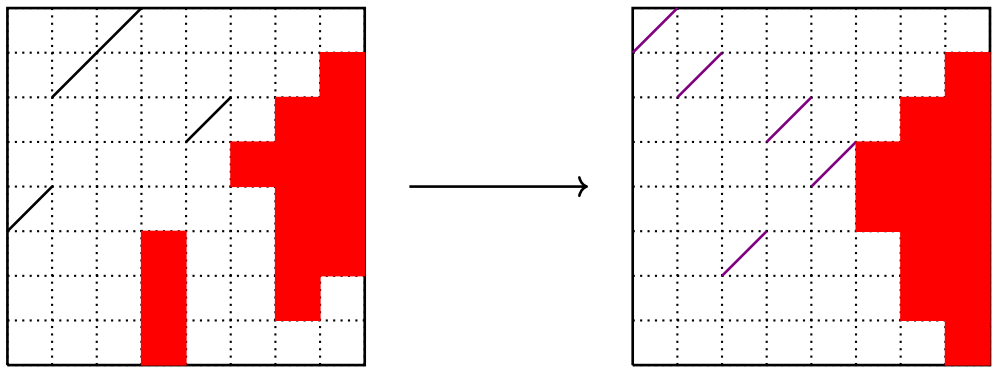
\begin{figure}[H]
    \centering
    \begin{tikzpicture}[scale=4.75]
    	\draw[-,line width=1] (0,0) -- (0,1) -- (1,1) -- (1,0) -- (0,0);
            %%%%%%%
            % Original shuffle
            
            % Support
    	\draw[-,line width=1] (3/8,0) -- (4/8,1/8);
            \draw[-,line width=1] (6/8,1/8) -- (8/8,3/8);
            \draw[-,line width=1] (5/8,4/8) -- (6/8,5/8);
            \draw[-,line width=1] (0,3/8) -- (1/8,4/8);
            \draw[-,line width=1] (1/8,6/8) -- (3/8,8/8);
            \draw[-,line width=1] (4/8,5/8) -- (5/8,6/8);

            % Support spiegeln
            
            % Grid
            \draw[step=1/8,black,thick,xshift=0cm,yshift=0cm,dotted] (0,0) grid (1,1);

            % Pfeil
            \draw[->,line width=1] (9/8,1/2) -- (13/8,1/2);
            
            %%%%%
            % Transformation
            \draw[-,line width=1] (14/8,0) -- (14/8,1) -- (22/8,1) -- (22/8,0) -- (14/8,0);

            \draw[step=1/8,black,thick,xshift=0cm,yshift=0cm,dotted] (14/8,0) grid (22/8,1);

            \draw[-,line width=1,violet] (21/8,0) -- (22/8,1/8);
            \draw[-,line width=1,violet] (20/8,1/8) -- (21/8,2/8);
            \draw[-,line width=1,violet] (19/8,3/8) -- (20/8,4/8);
            \draw[-,line width=1,violet] (18/8,4/8) -- (19/8,5/8);
            \draw[-,line width=1,violet] (16/8,2/8) -- (17/8,3/8);
            \draw[-,line width=1,violet] (17/8,5/8) -- (18/8,6/8);
            \draw[-,line width=1,violet] (15/8,6/8) -- (16/8,7/8);
            \draw[-,line width=1,violet] (14/8,7/8) -- (15/8,8/8);

            % Highlight shifts; lightgray
            \draw [fill=red, draw=red, nearly transparent] (3/8,0) rectangle (4/8, 3/8); 
            \draw [fill=red, draw=red, nearly transparent] (5/8,4/8) rectangle (6/8, 5/8);
            \draw [fill=red, draw=red, nearly transparent] (6/8,1/8) rectangle (7/8, 6/8);
            \draw [fill=red, draw=red, nearly transparent] (7/8,2/8) rectangle (8/8, 7/8);

            \draw [fill=red, draw=red, nearly transparent] (19/8,3/8) rectangle (20/8, 5/8);
            \draw [fill=red, draw=red, nearly transparent] (20/8,1/8) rectangle (21/8, 6/8);
            \draw [fill=red, draw=red, nearly transparent] (21/8,0/8) rectangle (22/8, 7/8);
    \end{tikzpicture}
    \caption{The mass rearrangement discussed in Example \ref{ex:rearrange}. Starting from the shuffle $S_\pi$ in the left
    panel we construct the shuffle $S_{\hat{\pi}} $ for which Spearman's footrule is the same
    but Spearman's $\rho$ is strictly smaller.}\label{MassReDistributionEx}
    \end{figure}   
\noindent \emph{As a direct consequence, $\phi(S_\pi)=\phi(S_{\hat{\pi}})=-\frac{5}{16}$ holds, 
i.e., Spearman's footrule does not change
when moving from $S_\pi$ to $S_{\hat{\pi}}$. Spearman's $\rho$, however does change, it decreases:  
in fact, we get 
$\rho(S_{\hat{\pi}})=1-\frac{3}{2} \cdot \frac{39}{32}=-0.828125 < -0.40625=1-\frac{3}{2}\cdot \frac{15}{16}=\rho(S_\pi)$. 
In terms of $m_{\pi,k}$ this translates to  
$$
m_{\pi,k}(\textbf{p}) >  m_{\hat{\pi},\hat{k}}(\hat{\textbf{p}})
$$
and for showing $m_{\pi,k}(\textbf{p}) \geq 0$ it suffices to show $m_{\hat{\pi},\hat{k}}(\hat{\textbf{p}}) \geq 0$. \\
The rearrangement from $S_\pi$ to $S_{\hat{\pi}}$ can be formalized as follows: 
Defining the functions $f:[0,1] \rightarrow [0,\infty)$ and $g:[0,1] \rightarrow \mathbb{R}$ by
\begin{eqnarray*}
f(x) &=& \frac{1}{2} \sum_{i=1}^4 p_i \mathbf{1}_{(\tfrac{i-1}{4},\tfrac{i}{4}]}(x), \quad 
g(x) = \frac{1}{2}\sum_{i=1}^4 q_i \mathbf{1}_{(\tfrac{i-1}{4},\tfrac{i}{4}]}(x),
\end{eqnarray*} 
with $(q_1,q_2,q_3,q_4)=\frac{1}{8}\cdot(1,1,0,-2)$, it follows that $p_1-q_1=0$ and 
$p_i - q_i = \hat{p}_{i-1}$ for $i \in \{2,3,4\}$. Obviously $f$ is non-decreasing whereas $g$ is non-increasing 
and fulfills $\int_{[0,1]} g(x) d\lambda(x)=\frac{1}{8} \sum_{i=1}^4 q_i=0$. 
Furthermore, letting $\Vert \cdot \Vert_2$ denote the $L_2$-norm with respect
to $\lambda$ on $\mathcal{B}([0,1])$ we have that  
\begin{eqnarray*}
\frac{1}{N}\sum_{i=1}^3 \hat{p_i}^2 &=&  \frac{1}{N}\sum_{i=1}^4 (p_i-q_i)^2 = \Vert f-g \Vert_2^2 \\
\frac{1}{N}\sum_{i=1}^4 p_i^2 &=&  \Vert f \Vert_2^2, 
\end{eqnarray*}
which yields the following equivalence:
\begin{eqnarray*}
\frac{1}{N}\sum_{i=1}^3 \hat{p_i}^2 \geq \frac{1}{N}\sum_{i=1}^4 p_i^2 \quad \textrm{if, and only if }  \quad
 \Vert f-g \Vert_2^2  \geq \Vert f \Vert_2^2.
\end{eqnarray*}
}
\end{Example}    
    
The last inequality $\Vert f-g \Vert_2^2  \geq \Vert f \Vert_2^2 $ 
turns out to be a (very) special case of a more general observation on 
$L_2$-norms of monotone functions $f,g$. Considering that the result may also be useful in the context of 
other problems we formulate and prove it directly for 
general finite measure spaces and a non-decreasing function $f$ in combination with a rearrangement function $g$ 
fulfilling the following (much weaker montonicity) property: 

\begin{Definition}
A function $g: [0,1] \rightarrow \mathbb{R}$ is called of decreasing block structure with respect to $0$ if
there exists some $x_0 \in [0,1]$ such that one of the following properties holds:
\begin{enumerate}
\item[(B1)] $g(x) \geq 0$ for every $x \in [0,x_0)$ and $g(x) \leq 0$ for every $x \in [x_0,1]$.
\item[(B2)] $g(x) \geq 0$ for every $x \in [0,x_0]$ and $g(x) \leq 0$ for every $x \in (x_0,1]$.
\end{enumerate}
\end{Definition}
Obviously every non-increasing $g: [0,1] \rightarrow \mathbb{R}$ is of decreasing block structure with respect to $0$.
The property $\int_{[0,1]} g d\mu=0$ is the reasons why we call $g$ a rearrangement function.  

\begin{Lemma}\label{LemmaResDis}
    Suppose that $\mu$ is a finite measure on $\mathcal{B}([0,1])$ and that $f,g \in L_2(\mu)$ fulfill the following properties: 
    \begin{enumerate}
        \item $f:[0,1] \rightarrow [0,\infty)$ is non-decreasing.
        \item $g:[0,1] \rightarrow \mathbb{R}$ has decreasing block structure with respect to $0$ and
        fulfills $\int_{[0,1]} g d\mu=0$.     
    \end{enumerate}
    Then the following inequality holds:
    \begin{equation}\label{ineq:rearrange.simple}
        \Vert f-g \Vert_2^2 \geq \Vert f \Vert_2^2 + \Vert g \Vert_2^2 
    \end{equation}  
\end{Lemma}

\begin{proof}
    First of all we obviously have $f-g \in L_2(\mu)$ as well as
    \begin{eqnarray*}
        \Vert f-g \Vert_2^2 &=& \Vert f \Vert_2^2 + \Vert g \Vert_2^2 - 2 \underbrace{\int_{[0,1]}fg d\mu}_{=:I},
    \end{eqnarray*} 
    so it suffices to show $I \leq 0$. We will prove the inequality for $g$ fulfilling property (B1) - 
    the case (B2) can be handled in the same manner.
    Using $0=\int_{[0,1]} g d\mu=\int_{[0,x_0)} g d\mu\, + \, \int_{[x_0,1]} g d\mu$ we have
   \begin{equation}\label{ineq:rearrange}
        \int_{[0,x_0)} g d\mu = - \int_{[x_0,1]} g d\mu = \int_{[x_0,1]} \underbrace{(-g)}_{\geq0} d\mu. 
    \end{equation}   
    Therefore, using monotonicity of $f$ it follows that
    \begin{align*}
        \int_{[0,x_0)} fg d\mu &\leq f(x_0) \int_{[0,x_0)} g d\mu = f(x_0)\int_{[x_0,1]} (-g) d\mu \leq 
        \int_{[x_0,1]} f(-g) d\mu \\&= - \int_{[x_0,1]} fg d\mu, 
    \end{align*}
    implying $I \leq 0$, and the proof is complete.
\end{proof}

%The proof shows that the result also holds for $g$ fulfilling $\int_{[0,1]} g d\mu = 0$ but not necessarily being monotone, 
%it suffices to have a value $x_0$ that can separate the non-negative and negative ‘blocks’ of $g$, i.e.,
%there exists some $x_0 \in [0,1]$ such that $g(x) \geq 0$ for every $x < x_0$ and 
%$g(x) <0 $ for every $x>x_0$. Will will refer to such $g$ as `fulfilling the block property'. Obviously
%every non-increasing $g: [0,1] \rightarrow \mathbb{R}$ fulfills the block property.
%Moreover, if $g$ only fulfills the second assertion of Lemma \ref{LemmaResDis} on a subinterval $[a,b] \subseteq [0,1]$ and is 
%$0$ on $[0,1] \setminus [a,b]$ then obviously ineq. (\ref{ineq:rearrange.simple}) still holds. This simple
%observation leads to the following generalization of the previous lemma:

The previous lemma can be extended to finite sums of function $g_i$ of block structure - the following 
general result holds:

\begin{Theorem}\label{thm:rearrange.general:sum}
Suppose that $\mu$ is a finite measure on $\mathcal{B}([0,1])$ and that the functions $f,g_1,\ldots,g_n \in L_2(\mu)$ 
fulfill the following properties: 
\begin{enumerate}
        \item $f:[0,1] \rightarrow [0,\infty)$ is non-decreasing.
        \item Each $g_i:[0,1] \rightarrow \mathbb{R}$ is of decreasing block structure with respect to $0$, 
        and fulfills $\int_{[0,1]} g_i d\mu=0$.
        \item Whenever $i \neq j$ we have $g_i(x) g_j(x)=0$ for every $x \in [0,1]$.            
    \end{enumerate}
    Then the following inequality holds for $g=\sum_{i=1}^n g_i$:
    \begin{equation}\label{ineq:rearrange.general}
        \Vert f-g \Vert_2^2 \geq \Vert f \Vert_2^2 + \Vert g \Vert_2^2 
    \end{equation}  

\end{Theorem}
\begin{proof}
From the proof of the previous lemma we know that $\int_{[0,1]}fg_i d\mu \leq 0$ for every $i \in \{1,\ldots,n\}$. 
Having that and considering 
\begin{eqnarray*}
\int_{[0,1]} \left(f- \sum_{i=1}^n g_i \right)^2 d\mu &=& \int_{[0,1]} f^2 d\mu + 
                            \int_{[0,1]} \underbrace{\left(\sum_{i=1}^n g_i \right)^2}_{=g^2} d\mu \\
                & & \, -2 \sum_{i=1}^n \underbrace{\int_{[0,1]} f g_i \, d\mu}_{\leq 0} \\
                &\geq& \int_{[0,1]} f^2 d\mu + \int_{[0,1]} g^2 d\mu 
\end{eqnarray*}
yields the desired result.
\end{proof}

\begin{Remark}
\emph{
Theorem  \ref{thm:rearrange.general:sum} can further be generalized since sets of $\mu$-measure $0$ can be ignored. 
However, for tackling the lower inequality we will work with $\mu=\lambda$, and the version stated above will suffice.
}
\end{Remark}

Before proceeding with the proof of the lower inequality we illustrate with an example, why working with 
sums $\sum_{i=1}^n g_i$ is necessary for proving the lower inequality.
\begin{Example}\label{ex:rearr2}
\emph{Consider $N = 16$ and the symmetric permutation 
$\pi = (15, 16, 3, 14, 11, 12, 7, 8, 9, 10, 5, 6, 13, 4, 1, 2) \in \Sigma_{16}$, 
for which we have $I_\pi^-=\{11, 12, 14, 15, 16\}, k = 5$, and $\mathbf{p} = \frac{1}{16}\cdot (6, 6, 10, 14, 14)$.
Following the rearrangement idea from the previous example yields the symmetric permutation 
$\hat{\pi}=(16,15,14,13,5,6,7,8,9,12,11,10,4,3,2,1)$ with $I_{\hat{\pi}}=\{12,13,14,15,16\}$ (see Figure \ref{Ex2}).  
Setting 
$$
\hat{\mathbf{p}} = \tfrac{1}{16} \cdot (2,9,11,13,15)
$$
and defining the functions $f, \hat{f}: [0,1] \rightarrow [0,\infty)$ and $g:[0,1] \rightarrow \mathbb{R}$ by 
\begin{eqnarray*}
f(x) &=& \frac{5}{16} \sum_{i=1}^5 p_i \mathbf{1}_{(\tfrac{i-1}{5},\tfrac{i}{5}]}(x), \quad 
\hat{f}(x) = \frac{5}{16}\sum_{i=1}^5 \hat{p_i} \mathbf{1}_{(\tfrac{i-1}{5},\tfrac{i}{5}]}(x), \quad g:=f-\hat{f}
\end{eqnarray*} 
yields
$$
g(x)=\tfrac{5}{16} \left( 4\cdot \mathbf{1}_{(0,\frac{1}{5}]}(x) - 3 \cdot \mathbf{1}_{(\frac{1}{5},\frac{2}{5}]}(x)
-1 \cdot \mathbf{1}_{(\frac{2}{5},\frac{3}{5}]}(x) + 1 \cdot \mathbf{1}_{(\frac{3}{5},\frac{4}{5}]}(x)
 - \mathbf{1}_{(\frac{4}{5},1]}(x) \right).
$$
Obviously $g$ fulfills $\int_{[0,1]} g \,d \lambda=0$ but is not of decreasing block structure w.r.t. $0$.
It is, however, the sum of following two functions $g_1,g_2$ of decreasing block structure and with integral $0$: 
\begin{eqnarray*}
g_1(x)&=& \tfrac{5}{16} \left( 4 \cdot \mathbf{1}_{[0,\frac{1}{5}]}(x) - 3 \cdot \mathbf{1}_{(\frac{1}{5},\frac{2}{5}]}(x)
-1 \cdot \mathbf{1}_{(\frac{2}{5},\frac{3}{5}]}(x) \right) \\
g_2(x) &=& \tfrac{5}{16} \left( 1 \cdot \mathbf{1}_{(\frac{3}{5},\frac{4}{5}]}(x)
 - \mathbf{1}_{(\frac{4}{5},1]}(x) \right).
\end{eqnarray*}   
As a direct consequence, ineq. (\ref{ineq:rearrange.general}) holds and we have 
$\Vert f-g \Vert_2^2 \geq \Vert f \Vert_2^2 + \Vert g \Vert_2^2$ with $g=g_1+g_2$, which in turn shows that 
$\rho(S_\pi) > \rho(S_{\hat{\pi}})$ as well as 
$$
m_{\pi,k}(\textbf{p}) >  m_{\hat{\pi},\hat{k}}(\hat{\textbf{p}})
$$
}
\end{Example} 

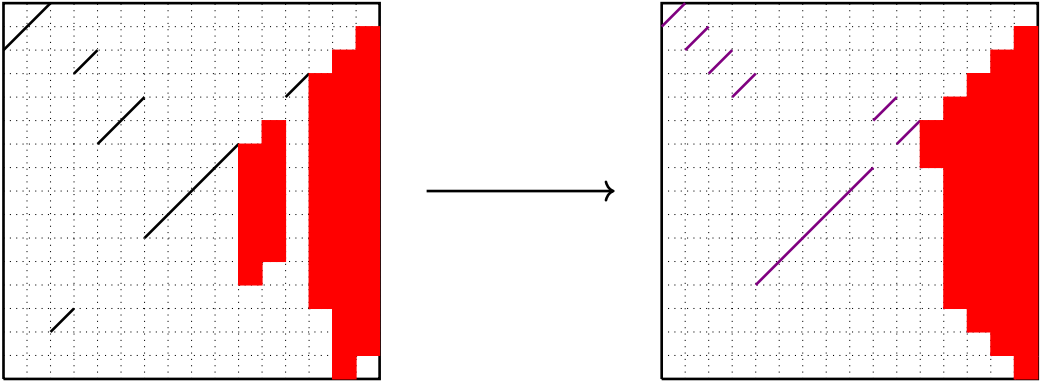
\begin{figure}[H]
        \centering
        \begin{tikzpicture}[scale=5]
        	\draw[-,line width=1] (0,0) -- (0,1) -- (1,1) -- (1,0) -- (0,0);
                \draw[step=1/16,black,thin,xshift=0cm,yshift=0cm,dotted] (0,0) grid (1,1);
    
                % Mass on diagonal
                \draw[-,line width=1] (2/16,2/16) -- (3/16,3/16);
                \draw[-,line width=1] (6/16,6/16) -- (10/16,10/16);
                \draw[-,line width=1] (12/16,12/16) -- (13/16,13/16);
    
                % Mass below diagonal
                \draw[-,line width=1] (10/16,4/16) -- (12/16,6/16);
                \draw[-,line width=1] (13/16,3/16) -- (14/16,4/16);
                \draw[-,line width=1] (14/16,0/16) -- (16/16,2/16);
    
                % Mass above diagonal
                \draw[-,line width=1] (4/16,10/16) -- (6/16,12/16);
                \draw[-,line width=1] (3/16,13/16) -- (4/16,14/16);
                \draw[-,line width=1] (0/16,14/16) -- (2/16,16/16);
    
                % Highlight shifts;
                \draw [fill=red, draw=red, nearly transparent] (10/16,4/16) rectangle (11/16, 10/16); 
                \draw [fill=red, draw=red, nearly transparent] (11/16,5/16) rectangle (12/16, 11/16); 
                \draw [fill=red, draw=red, nearly transparent] (13/16,3/16) rectangle (14/16, 13/16); 
                \draw [fill=red, draw=red, nearly transparent] (14/16,0/16) rectangle (15/16, 14/16); 
                \draw [fill=red, draw=red, nearly transparent] (15/16,1/16) rectangle (16/16, 15/16);
    
                % Pfeil
                \draw[->,line width=1] (9/8,1/2) -- (13/8,1/2);
    
                 % Transformation
                \draw[-,line width=1] (28/16,0) -- (28/16,1) -- (44/16,1) -- (44/16,0) -- (28/16,0);
    
                \draw[step=1/16,black,thin,xshift=0cm,yshift=0cm,dotted] (28/16,0) grid (44/16,1);
    
                % Mass on diagonal
                \draw[-,line width=1, violet] (32/16,4/16) -- (37/16,9/16);
                \draw[-,line width=1, violet] (38/16,10/16) -- (39/16,11/16);
    
                % Mass below diagonal
                \draw[-,line width=1, violet] (39/16,9/16) -- (40/16,10/16);
                \draw[-,line width=1, violet] (40/16,3/16) -- (41/16,4/16);
                \draw[-,line width=1, violet] (41/16,2/16) -- (42/16,3/16);
                \draw[-,line width=1, violet] (42/16,1/16) -- (43/16,2/16);
                \draw[-,line width=1, violet] (43/16,0/16) -- (44/16,1/16);
    
                % Mass above diagonal
                \draw[-,line width=1, violet] (37/16,11/16) -- (38/16,12/16);
                \draw[-,line width=1, violet] (31/16,12/16) -- (32/16,13/16);
                \draw[-,line width=1, violet] (30/16,13/16) -- (31/16,14/16);
                \draw[-,line width=1, violet] (29/16,14/16) -- (30/16,15/16);
                \draw[-,line width=1, violet] (28/16,15/16) -- (29/16,16/16);
    
                % Highlight shifts;
                \draw [fill=red, draw=red, nearly transparent] (43/16,0/16) rectangle (44/16, 15/16);
                \draw [fill=red, draw=red, nearly transparent] (42/16,1/16) rectangle (43/16, 14/16);
                \draw [fill=red, draw=red, nearly transparent] (41/16,2/16) rectangle (42/16, 13/16);
                \draw [fill=red, draw=red, nearly transparent] (40/16,3/16) rectangle (41/16, 12/16);
                \draw [fill=red, draw=red, nearly transparent] (39/16,9/16) rectangle (40/16, 11/16);
        \end{tikzpicture}
        \caption{The two shuffles considered in Example \ref{ex:rearr2}}\label{Ex2}
    \end{figure}

We finally return to proving ineq. \eqref{LowerBound} for symmetric shuffles we apply Lemma \ref{LemmaResDis} and  
proceed in two steps: (i) Using the rearrangement idea reduce the problem to a handy subclass of symmetric shuffles.
(ii) Prove the ine\-quality for all elements of the subclass and again use the fact, that symmetric shuffles 
are dense (Lemma \ref{DenseSymShuffles}). 

The subclass consists of all symmetric shuffles/permutations of the type $\hat{\pi}$ considered in the two 
previous examples. We will define these permutations $\pi$ only on the sets $I_\pi^-$ since, using 
symmetry their extension to $I_\pi^+$ and $I_\pi^0$ is unique.   
In what follows, $N\in\N$ will be arbitrary but fixed. For each such $N$ define
\begin{align*}
    \hat{\Sigma}_N^1 &:= \{\hat{\pi} \in \Sigma_N: \hat{\pi}\text{ symmetric, } I_{\hat{\pi}}^- = \{N\} \text{ and } \hat{\pi}(N) \in \{1, \dots, N-1\} \} \\
    \hat{\Sigma}_N^2 &:= \Big\{\hat{\pi} \in \Sigma_N: \hat{\pi} \text{ symmetric, } \hat{k} \geq 2, 
    I_{\hat{\pi}}^- =  \{N - \hat{k} + 1, N - \hat{k} + 2 ,\ldots, N\} \\&
    \quad \quad \text{ such that } \hat{\pi}(N - \hat{k} + 1) \in \{1, \dots, N - (2\hat{k}-1)\} \text{ and }
     \\& \quad \quad \, \hat{\pi}(N - i + 1) = i \text{ for } i \in \{1, \dots, \hat{k} -1 \} \Big\},
\end{align*}
set $\hat{\Sigma}_N := \hat{\Sigma}_N^1 \cup \hat{\Sigma}_N^2$, and let $\hat{\mathcal{S}}_N, \hat{\mathcal{S}}_N^1, 
 \hat{\mathcal{S}}_N^2 \subseteq \mathcal{S}_N^{sym}$ denote the corresponding families of shuffles. 
The following result formalizes the afore-mentioned reduction and shows that it suffices to prove  
ineq. \eqref{LowerBound} for all shuffles in $\hat{\mathcal{S}}_N$.

\begin{Lemma}
    Suppose that $N \in \N$ with $N\geq4$ and that $S_\pi \in \mathcal{S}_N^{sym}$. Then there exists some  
    shuffle $S_{\hat{\pi}} \in \hat{\mathcal{S}}_N$ such that 
    $m_{\pi, k}(\textbf{p}) \geq m_{\hat{\pi}, \hat{k}}(\hat{\textbf{p}})$ holds.
\end{Lemma}

\begin{proof}
    Let $N$ be as in the theorem, $\pi \in \Sigma_N^{sym}$ be arbitrary but fixed, and $k = \# I_\pi^-$. 
     Setting 
    \begin{align*}
        f(x) = \frac{k}{N} \sum_{i=1}^k p_i \mathbf{1}_{(\tfrac{i-1}{k},\tfrac{i}{k}]}(x)
    \end{align*}
    with $\textbf{p} \in [0,1]^k$ corresponding to the shuffle $S_\pi$, it follows that $f$ is non-decreasing.
    Writing $\Delta = N \sum_{i=1}^k p_i$ define 
    \begin{align*}
        \ell := 
        \begin{cases}
        0, &\text{if } \Delta < N - 1, \\
        \max\{ j \in \{1, \dots, k\}: \sum_{i=1}^j N - (2i - 1) \leq \Delta \}, &\text{else.} 
    \end{cases}
    \end{align*} 
    (i) In case of $\ell = 0$ we consider the symmetric permutation 
    $\hat{\pi} \in \hat{\Sigma}_N^1$ with $I_{\hat{\pi}}^- = \{N\}$ and $\hat{\pi}(N) := N - \Delta$. 
    Then $\hat{p} = \frac{\Delta}{N}$ and setting 
    \begin{align*}
        \hat{f}(x) = \frac{k}{N}\; \hat{p} \; \mathbf{1}_{(\tfrac{k-1}{k},1]}(x)
    \end{align*}
    yields that $g := f - \hat{f}$ satisfies $\int_{[0,1]} g d \lambda = 0$ and $g$ is of decreasing block 
    structure w.r.t. $0$. Applying Lemma \ref{LemmaResDis} yields $\Vert f - g \Vert_2^2 \geq \Vert f \Vert_2^2$ 
    or, equivalently, $\rho(S_{\pi}) \geq \rho(S_{\hat{\pi}})$. By construction we have $\phi(S_{\pi}) = \phi(S_{\hat{\pi}})$ and therefore $m_{\pi, k}(\textbf{p}) \geq m_{\hat{\pi}, 1}(\hat{p})$ holds.
    
    (ii) Let $\ell \geq 1$ and set $\Delta^\ast = \sum_{i=1}^\ell N - (2i - 1)$. 
    (a) Suppose that $\Delta - \Delta^\ast >0$. Setting $\hat{k} := \ell + 1 \leq k$ we consider $\hat{\pi} \in \Sigma_N$ with $I_{\hat{\pi}}^- = \bigcup_{i=1}^{\hat{k}} \{N - i + 1 \}$, $\hat{\pi}(N - \hat{k} + 1) = N - \hat{k} + 1 - (\Delta - \Delta^\ast)$ and $\hat{\pi}(N - i + 1) = i$ for $i \in \{1, \dots, \hat{k} -1 \}$. 
    Then $\hat{\pi} \in \hat{\Sigma}_N^2$ with $\hat{\textbf{p}} = \left( \tfrac{\Delta - \Delta^\ast}{N}, \tfrac{N - (2\ell-1)}{N}, \dots, \tfrac{N-1}{N}\right)$ and, by construction, we have $\phi(S_{\pi}) = \phi(S_{\hat{\pi}})$. Setting 
    \begin{align*}
        \hat{f}(x) = \frac{k}{N} \sum_{i=k - \hat{k} + 1}^k \hat{p}_i \mathbf{1}_{(\tfrac{i-1}{k},\tfrac{i}{k}]}(x)
    \end{align*}
    yields that $g := f - \hat{f}$ satisfies $\int_{[0,1]} g d \lambda = 0$. Considering that $g$ 
    can be expressed as finite sum of functions $g_i$ of decreasing block structure w.r.t. $0$ and $g_i g_j=0$ 
    whenever $i \neq j$, applying Theorem \ref{thm:rearrange.general:sum} yields 
    $\Vert f - g \Vert_2^2 \geq \Vert f \Vert_2^2$, i.e., $\rho(S_{\pi}) \geq \rho(S_{\hat{\pi}})$ and 
    $m_{\pi, k}(\textbf{p}) \geq m_{\hat{\pi}, \hat{k}}(\hat{\textbf{p}})$. \\    
    (b) Finally, suppose that $\Delta - \Delta^\ast = 0$. Set $\hat{k} := \ell$ and consider 
    $\hat{\pi} \in \Sigma_N$ with $I_{\hat{\pi}}^- = \bigcup_{i=1}^{\hat{k}} \{N - i + 1 \}$ and $\hat{\pi}(N - i + 1) = i$ 
    for $i \in \{1, \dots, \hat{k}\}$. Then $\hat{\pi} \in \hat{\Sigma}_N^2$ with 
    $\hat{\textbf{p}} = \left( \tfrac{N - (2k-1)}{N}, \dots, \tfrac{N-1}{N}\right)$. Setting
    \begin{align*}
        \hat{f}(x) = \frac{k}{N} \sum_{i=1}^k \hat{p}_i \mathbf{1}_{(\tfrac{i-1}{k},\tfrac{i}{k}]}(x).
    \end{align*}
    as well as $g := f - \hat{f}$ and proceeding analogously to the previous case concludes the proof. 
\end{proof} 

As second and ultimate step we show that ineq. \eqref{LowerBound} holds for all shuffles 
within the subclass $\hat{\mathcal{S}}_N$.
To this end, define $M_{N,k} := \{\textbf{x} \in [0,1]^k: \frac{1}{N} \sum_{i=1}^k x_i \leq 1/4 \}$ and interpret 
$m_{\pi,k}$ as function on $M_{N,k}$. 

\begin{Theorem}\label{LowerboundShuffleSubclass}
    For every shuffle $S_{\hat{\pi}} \in \hat{\mathcal{S}}_N$ with resolution $N\in\N$ and $N\geq 4$ the following inequality holds:
    \begin{align*}
        \rho(S_{\hat{\pi}}) \geq \frac{2}{9} \sqrt{3} \left( 1 + 2\phi(S_{\hat{\pi}}) \right)^{3/2} - 1.
    \end{align*}
\end{Theorem}

\begin{proof}
    We show that under the assumptions of the theorem 
    $m_{\hat{\pi}, \hat{k}}(\hat{\textbf{p}}) \geq 0$ holds and proceed as follows: 
    For $\textbf{x} \in M_{N, \hat{k}}$, calculating the first and second partial derivative w.r.t. $x_j$ 
    yields
    \begin{align*}
        \frac{\partial}{\partial x_j} m_{\hat{\pi}, \hat{k}}(\textbf{x}) &= - \frac{12}{N} x_j + \frac{6}{N} \sqrt{1 - \frac{4}{N} \sum_{i = 1}^{\hat{k}} x_i} = 
        \frac{6}{N} \left( \sqrt{1 - \frac{4}{N} \sum_{i = 1}^{\hat{k}} x_i} - 2x_j \right)
    \end{align*}
    as well as
    \begin{align*}
          \frac{\partial^2}{\partial^2 x_j} m_{\hat{\pi}, \hat{k}}(\textbf{x}) = - \frac{12}{N} \left( \frac{1}{N \sqrt{1 - \frac{4}{N} \sum_{i = 1}^{\hat{k}} x_i}} + 1 \right) < 0
    \end{align*}
    for every $j \in \{1, \dots, k\}$. As a direct consequence, 
    the mapping $$t \mapsto m_{\hat{\pi}, \hat{k}}(x_1, \dots, x_{j-1}, t, x_{j+1}, \dots, x_k) $$ is 
    concave on the convex polytope $M_{N,\hat{k}}$. \\
    We again distinguish two cases: (i) if $S_{\hat{\pi}} \in \hat{\mathcal{S}}_N^1$ then $\hat{k} = 1$ and concavity of $m_{\hat{\pi}, 1}$ implies that $m_{\hat{\pi}, 1}(\hat{p}) \geq m_{\hat{\pi}, 1}(0) = 0$ or $m_{\hat{\pi}, 1}(\hat{p}) \geq m_{\hat{\pi}, 1}(\tfrac{N-1}{N})$. A straightforward calculation yields
    \begin{align*}
        N^3 m_{\hat{\pi}, 1}(\tfrac{N-1}{N}) &= N^3 \left( 1 - \frac{6(N-1)^2}{N^3} - \left(1 - \frac{4(N-1)}{N^2}\right)^{3/2} \right) \\&= N^3 - 6(N-1)^2 - \left(N^2 - 4(N-1) \right)^{3/2} \\&=
        N^3 - 6(N-1)^2 - (N-2)^3 = 2,
    \end{align*}
    implying the assertion for $k=1$. \\
    (ii) In case that $S_{\hat{\pi}} \in \hat{\mathcal{S}}_N^2$ we have $\hat{\textbf{p}} = \left( \tfrac{j}{N}, \tfrac{N - (2(\hat{k}-1)-1)}{N}, \dots, \tfrac{N-1}{N}\right)$ with $j \in \{1, \dots, N - (2\hat{k}-1)\}$. Now, setting 
    \begin{align*}
        \hat{\textbf{p}}_0 &:= \left( 0, \tfrac{N - (2(\hat{k}-1)-1)}{N}, \dots, \tfrac{N-1}{N}\right) \in M_{N,\hat{k}} \\
        \hat{\textbf{p}}_1 &:= \left( \tfrac{N - (2\hat{k}-1)}{N}, \tfrac{N - (2(k-1)-1)}{N}, \dots, \tfrac{N-1}{N}\right) \in M_{N,\hat{k}}
    \end{align*}
    concavity of $m_{\hat{\pi}, \hat{k}}$ in each coordinate implies that $m_{\hat{\pi}, \hat{k}}(\hat{\textbf{p}}) \geq m_{\hat{\pi}, \hat{k}}(\hat{\textbf{p}}_j)$ holds for $j \in \{0, 1\}$. 
    Therefore, considering $\hat{k}_j := \hat{k} + j - 1 $ and $\hat{k}_j \leq N/2$ it follows that 
    \begin{align*}
         1 - \frac{4}{N} \sum_{i = 1}^{\hat{k}} \hat{p}_{ji} = 1 - \frac{4}{N}  \sum_{i=1}^{\hat{k}_j} \frac{N - (2i - 1)}{N} = \frac{(N-2\hat{k}_j)^2}{N^2}.
    \end{align*}
    as well as 
    \begin{align*}
        1 - \frac{6}{N} \sum_{i = 1}^{\hat{k}} \hat{p}_{ji}^2 > 1 - \frac{6}{N} \sum_{i=1}^{\hat{k}_j} \frac{(N - (2i - 1))^2}{N^2} - \frac{2\hat{k}_j}{N^3} = \frac{(N-2\hat{k}_j)^3}{N^3}.
    \end{align*}
    Therefore $m_{\hat{\pi}, \hat{k}}(\hat{\textbf{p}}_j) > 0$ and the proof is complete. 
\end{proof}

Following the idea of the proof of Theorem \ref{LowerboundShuffleSubclass} the copula $C_\alpha$ depicted in Figure \ref{LowerBoundSharp} for $\alpha \in [0,\frac{1}{2}]$ seems a very natural suspect for attaining the lower bound.

\begin{figure}[H]
    \centering
    \begin{subfigure}{.5\textwidth}
        \centering
        \begin{tikzpicture}[scale=4]
            \draw[-,line width=1] (0,0) -- (0,1) -- (1,1) -- (1,0) -- (0,0);
            % Support
            \draw[-,line width=1] (4/12, 4/12) -- (8/12, 8/12);
            \draw[-,line width=1] (0,1) -- (4/12, 8/12);
            \draw[-,line width=1] (8/12, 4/12) -- (1,0);
    	% x-Axis label
            \draw[-,line width=1] (4/12,0) -- (4/12,-0.02);
    	\node[below=1pt of {(4/12,-0.02)}, outer sep=2pt,fill=white]{$\alpha$};
            \draw[-,line width=1] (8/12,0) -- (8/12,-0.02);
    	\node[below=1pt of {(8/12,-0.02)}, outer sep=2pt,fill=white]{$1-\alpha$};
    	% y-Axis label
            \draw[-,line width=1] (0,4/12) -- (-0.02,4/12);
    	\node[left=1pt of {(-0.02,4/12)}, outer sep=2pt,fill=white] {$\alpha$};
            \draw[-,line width=1] (0,8/12) -- (-0.02,8/12);
    	\node[left=1pt of {(-0.02,8/12)}, outer sep=2pt,fill=white] {$1-\alpha$};
	\end{tikzpicture}
    \end{subfigure}%
    %\hspace{0.25cm}
    \begin{subfigure}{.5\textwidth}
        \centering
        \begin{tikzpicture}[scale=4]
    	\draw[-,line width=1] (0,0) -- (4/12,0) -- (8/12, 4/12) -- (1,1);
    	\draw[->,line width=1] (-0.1,0) -- (1.1,0);
            \draw[->,line width=1] (0,-0.1) -- (0,1.1);
            \draw[-,line width=1, dotted] (0,0) -- (1,1);
            % x-Axis label
	    \draw[-,line width=1] (4/12,0) -- (4/12,-0.02);
	    \node[below=1pt of {(4/12,-0.02)}, scale= 0.75, outer          sep=2pt,fill=white] {$a$};
            \draw[-,line width=1] (1/2,0) -- (1/2,-0.02);
	    \node[below=1pt of {(1/2,-0.02)}, scale= 0.75, outer          sep=2pt,fill=white] {$\frac{1}{2}$};
        \end{tikzpicture}
    \end{subfigure}
    \caption{The copula $C_\alpha$ (left panel) for which the lower inequality is sharp, and its corresponding 
    diagonal $\delta_\alpha$ (right panel).}\label{LowerBoundSharp}
\end{figure}
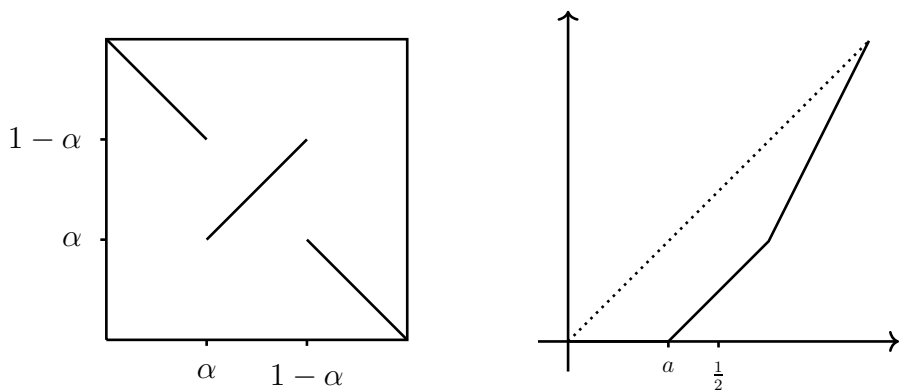
In fact, it was already shown in \cite{bukovvsek2024exact} that for this copula
\begin{align*}
    \phi(C_\alpha) &= 6\alpha^2 - 6\alpha + 1,\\
    \rho(C_\alpha) &= -16\alpha^3 + 24\alpha^2 - 12\alpha + 1.
\end{align*}
holds, so $ \rho(C_\alpha) = \frac{2}{9} \sqrt{3} \left( 1 + 2\phi(C_\alpha) \right)^{3/2} - 1 $ 
for every $\alpha \in [0,\frac{1}{2}]$.

\clearpage
\section{Getting closer to the non-sharp upper bound}\label{SecImprovementUpperBound}

We now return to the upper bound for $\Omega_{\phi,\rho}$ which was already shown not to be (globally) sharp 
by the authors in \cite{bukovvsek2024exact}. The natural question therefore is, how 
non-sharp it is, i.e., how close one can get to the upper bound. In \cite{bukovvsek2024exact}
the authors also tackled this question, derived the function $r: [-\frac{1}{2},1] \rightarrow [-1,1]$, defined by
\begin{align*}
    r(x) =
    \begin{cases}
        2x + \frac{1}{2} - \frac{\sqrt{3}}{9} (1+2x)^{\frac{3}{2}}, &\text{if } x \in[-\frac{1}{2}, -\frac{1}{8}] \\
        \frac{4}{3}x + \frac{7}{24}, &\text{if } x  \in[-\frac{1}{8},\frac{1}{4}] \\
        \frac{2n+1}{n^2+n}x + \frac{2n^2-2n+1}{2n^2+2n} &\text{if } x \in[1-\frac{3}{2n}, 1-\frac{3}{2(n+1)}] \\
        1  &\text{if } x =1
    \end{cases}
\end{align*}
and showed that for every point $(x,r(x))$ there exists some copula $C$ with $\phi(C)=x$ and $\rho(C)=r(x)$. 
Notice that $r$ is piecewise linear on the interval $[-\frac{1}{8},1]$ and is, for arbitrary $x \in [-\frac{1}{2},1]$ 
quite close to the upper bound given by $1-\frac{2}{3}(1-x)^2$. 

The goal of this section is to prove that on the interval $[-\frac{1}{8},1]$ - outside the countably many points 
on which the upper inequality is known to be sharp - the function $r$ can be exceeded. 
Using convexity of $\Omega_{\phi,\rho}$ we will work with even $n \in \mathbb{N}$ and the 
shuffles $S_n^\ast:=S_{\pi^\ast}$ with $\pi^\ast \in \Sigma_n$ fulfilling  
$i - \pi^\ast(i) = 1$ for every $i \in I_{\pi^\ast}^- =  \{i \in \{1, \dots, n\}: i \text{ even}\}$. 
As mentioned in Section 3, these shuffles
constitute points at which the upper inequality is sharp (see Figure \ref{FigShuffleUpperBound} for an example). 
Our idea consists in `interpolating' between two consecutive shuffles $S_n^\ast$ and $S_{n+2}^\ast$.
Illustrating the approach we start with the pair $S_2^\ast$ and $S_4^\ast$ and then extend to 
$S_n^\ast$ and $S_{n+2}^\ast$ for arbitrary even $n \in \mathbb{N}$.

\subsection{Special case: Interpolating between $S_2^\ast$ and $S_4^\ast$}
We show that for every $x \in (-\frac{1}{8}, \frac{1}{4})$ there exists some copula $C$ with $\phi(C)=x$
fulfilling that $\rho(C) > r(x)$ and proceed as follows:

Given $a\in[\frac{1}{4}, \frac{1}{2}]$ and $b\in[0, \frac{1}{4}]$ define the two diagonals $\delta_a^\uparrow$ 
and $\delta_b^\downarrow$ by (see Figure \ref{Interpol1} and \ref{Interpol2} for an illustration) 
\begin{align*}
    \delta_a^\uparrow(x) =
    \begin{cases}
        0, &\text{if } x \in[0, a] \\
        x-a, &\text{if } x \in[a, 1-a] \\
        2x-1 &\text{if } x \in[1-a, 1]
    \end{cases}
\end{align*}
and 
\begin{align*}
     \delta_b^\downarrow(x) =
    \begin{cases}
        0, &\text{if } x \in[0, \frac{1}{4}] \\
        2(x-\frac{1}{4}), &\text{if } x \in[\frac{1}{4}, \frac{1}{4} + b] \\
        x + b - \frac{1}{4}, &\text{if } x \in[\frac{1}{4} + b, \frac{3}{4} - b] \\
        \frac{1}{2}, &\text{if } x \in[\frac{3}{4} - b, \frac{3}{4}] \\
        2x-1 &\text{if } x \in[\frac{3}{4}, 1]
    \end{cases}
\end{align*}
and consider the corresponding diagonal copulas $E_{\delta_a^\uparrow}$ and $E_{\delta_b^\downarrow}$

\begin{figure}[H]
    \centering
    \begin{subfigure}{.5\textwidth}
        \centering
        \begin{tikzpicture}[scale=4]
    	\draw[-,line width=1] (0,0) -- (4/12,0) -- (8/12, 4/12) -- (1,1);
    	\draw[->,line width=1] (-0.1,0) -- (1.1,0);
            \draw[->,line width=1] (0,-0.1) -- (0,1.1);
            \draw[-,line width=1, dotted] (0,0) -- (1,1);
            % x-Axis label
	    \draw[-,line width=1] (4/12,0) -- (4/12,-0.02);
	    \node[below=1pt of {(4/12,-0.02)}, scale= 0.75, outer          sep=2pt,fill=white] {$a$};
            \draw[-,line width=1] (1/2,0) -- (1/2,-0.02);
	    \node[below=1pt of {(1/2,-0.02)}, scale= 0.75, outer          sep=2pt,fill=white] {$\frac{1}{2}$};
        \end{tikzpicture}
    \end{subfigure}%
    %\hspace{0.25cm}
    \begin{subfigure}{.5\textwidth}
        \centering
        \begin{tikzpicture}[scale=4]
    	\draw[-,line width=1] (0,0) -- (0,1) -- (1,1) -- (1,0) -- (0,0);
            % Support transformation
            \draw[-,line width=1, black] (0,4/12) -- (2/12,8/12) -- (4/12,10/12) -- (8/12, 1);
            \draw[-,line width=1, black] (4/12,0) -- (8/12,2/12) -- (10/12,4/12) -- (1, 8/12);
            % x-Axis label
	    \draw[-,line width=1] (4/12,0) -- (4/12,-0.02);
	    \node[below=1pt of {(4/12,-0.02)}, scale= 0.75, outer                       sep=2pt,fill=white] {$a$};
            \draw[-,line width=1] (1/2,0) -- (1/2,-0.02);
	    \node[below=1pt of {(1/2,-0.02)}, scale= 0.75, outer                       sep=2pt,fill=white] {$\frac{1}{2}$};
            % y-Axis label
            \draw[-,line width=1] (0, 4/12) -- (-0.02,4/12);
	    \node[left=1pt of {(-0.02, 4/12)}, scale= 0.75, outer                       sep=2pt,fill=white] {$a$};
	\end{tikzpicture}
    \end{subfigure}
    \caption{The diagonal $\delta_a^\uparrow$ (left panel) and the support of the corresponding 
    diagonal copula $E_{\delta_a^\uparrow}$ (right panel).}\label{Interpol1}
\end{figure}
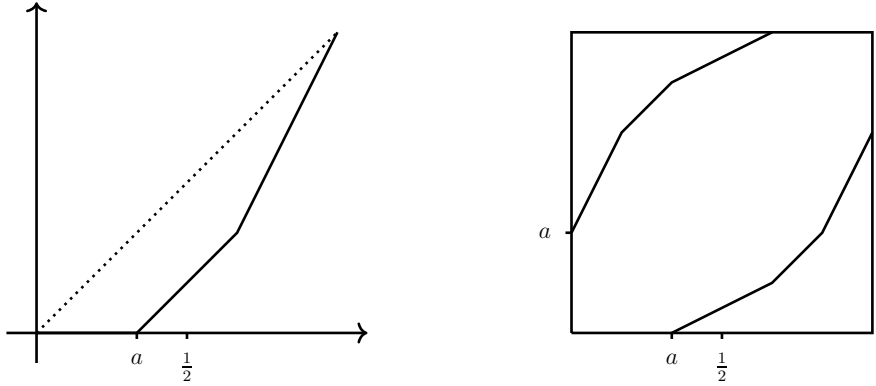

\begin{figure}[H]
    \centering
    \begin{subfigure}{.5\textwidth}
        \centering
        \begin{tikzpicture}[scale=4]
    	\draw[-,line width=1] (0,0) -- (1/4,0) -- (4/12, 2/12) -- (8/12,6/12) -- (9/12, 6/12) -- (1,1);
    	\draw[->,line width=1] (-0.1,0) -- (1.1,0);
            \draw[->,line width=1] (0,-0.1) -- (0,1.1);
            \draw[-,line width=1, dotted] (0,0) -- (1,1);
            % x-Axis label
	    \draw[-,line width=1] (4/12,0) -- (4/12,-0.02);
	    \node[below=1pt of {(9/24,-0.02)}, scale= 0.75, outer sep=2pt] {$\frac{1}{4} + b$};
            \draw[-,line width=1] (1/4,0) -- (1/4,-0.02);
	    \node[below=1pt of {(1/4,-0.02)}, scale= 0.75, outer sep=2pt] {$\frac{1}{4}$};
        \end{tikzpicture}
    \end{subfigure}%
    %\hspace{0.25cm}
    \begin{subfigure}{.5\textwidth}
        \centering
        \begin{tikzpicture}[scale=4]
    	\draw[-,line width=1] (0,0) -- (0,1) -- (1,1) -- (1,0) -- (0,0);
            % Support transformation
            \draw[-,line width=1, black] (0,1/4) -- (1/12, 4/12) -- (1/4, 8/12);
            \draw[-,line width=1, black] (1/4,0) -- (4/12, 1/12) -- (8/12, 1/4);
            \draw[-,line width=1, black] (4/12,9/12) -- (8/12,11/12) -- (9/12, 1);
            \draw[-,line width=1, black] (9/12,4/12) -- (11/12,8/12) -- (1,9/12);
          % x-Axis label
	    \draw[-,line width=1] (4/12,0) -- (4/12,-0.02);
	    \node[below=1pt of {(9/24,-0.02)}, scale= 0.75, outer sep=2pt] {$\frac{1}{4} + b$};
            \draw[-,line width=1] (1/4,0) -- (1/4,-0.02);
	    \node[below=1pt of {(1/4,-0.02)}, scale= 0.75, outer sep=2pt] {$\frac{1}{4}$};
     \draw[-,line width=1] (1/12,0) -- (1/12,-0.02);
	    \node[below=1pt of {(1/12,-0.02)}, scale= 0.75, outer sep=2pt,fill=white] {$b$};
	\end{tikzpicture}
    \end{subfigure}
    \caption{The diagonal $\delta_b^\downarrow$ (left panel) and the support of the induced diagonal copula $E_{\delta_b^\downarrow}$ (right panel).}\label{Interpol2}
\end{figure}
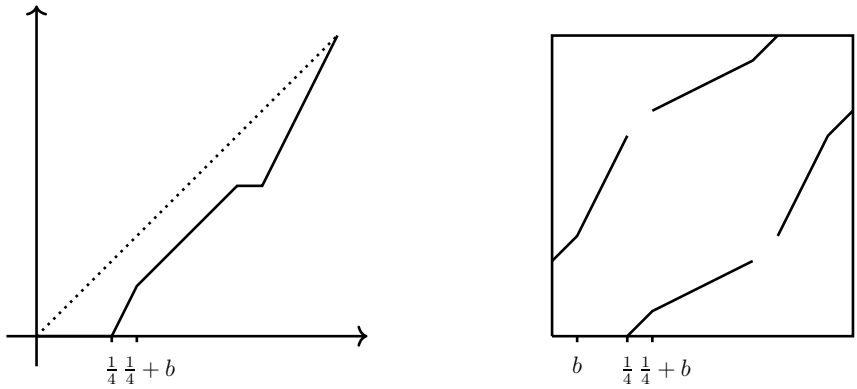
\noindent Note that $\delta_a^\uparrow$ coincides with the diagonal $\delta_a$ studied in Example 14 in 
\cite{bukovvsek2024exact} (and the copula $K_{\delta_a}$ considered there with the
diagonal copula $E_{\delta_a^\uparrow}$), so we obtain we obtain
we obtain
\begin{align*}
    \phi(E_{\delta_a^\uparrow}) &= 6a^2 - 6a + 1 \\
    \rho(E_{\delta_a^\uparrow}) &= 8a^3 - 6a + \frac{3}{2}
\end{align*}
as well as $\rho(E_{\delta_a^\uparrow}) = 2\phi(E_{\delta_a^\uparrow}) + \frac{1}{2} - \frac{\sqrt{3}}{9} (1 + 2\phi(E_{\delta_a^\uparrow}))^{3/2} $. Furthermore, tedious but simple calculations (see Appendix) yields  
\begin{align*}
    \phi(E_{\delta_b^\downarrow}) &= -6b^2 + 3b - \frac{1}{8}  \\
    \rho(E_{\delta_b^\downarrow}) &= 8b^3 - 12b^2 + \, \frac{9}{2}b + \frac{1}{8}
\end{align*}
which altogether implies $\rho(E_{\delta_b^\downarrow}) = \phi(E_{\delta_b^\downarrow}) + \frac{3}{8} - \frac{\sqrt{6}}{36} (1 - 4 \phi(E_{\delta_b^\downarrow}))^{3/2}.$

Hence, defining the concave function $h: [-\frac{1}{8}, \frac{1}{4}] \mapsto [0,1]$ by 
$$h(x) := \frac{3}{8} +  x - \frac{\sqrt{6}}{36} (1-4x)^{3/2}$$ 
we obtain $\rho(E_{\delta_b^\downarrow}) = h(\phi(E_{\delta_b^\downarrow}))$. Varying the values of $a \in [\frac{1}{4}, \frac{1}{2}]$ and $b \in [0, \frac{1}{4}]$, the copulas $E_{\delta_a^\uparrow}$ and $E_{\delta_b^\downarrow}$ yield an `interpolation' from the copula $S_2^\ast$ to $S_4^\ast$ in the following sense:
\begin{align*}
    S_2^\ast = E_{\delta_{1/2}^\uparrow} \overset{a \to 1/4}{\longrightarrow} E_{\delta_{1/4}^\uparrow} = E_{\delta_{0}^\downarrow} \overset{b\to1/4}{\longrightarrow} E_{\delta_{1/4}^\downarrow} = S_4^\ast
\end{align*}
Recall that for $x \in [-\frac{1}{8}, \frac{1}{4}]$ the graph of function $r$ is the linear interpolation 
between $(\phi(E_{\delta_{0}^\downarrow}), \rho(E_{\delta_{0}^\downarrow}))$ and $(\phi(S_4^\ast), \rho(S_4^\ast))$. 
Considering $\phi(E_{\delta_{0}^\downarrow}) = - \frac{1}{8}$ and $\phi(E_{\delta_{1/4}^\downarrow}) = \frac{1}{4}$, 
strict concavity of $h$ implies that $h(x) > r(x)$ for $x \in (-\frac{1}{8}, \frac{1}{4})$,
so the afore-mentioned construction exceeds $r$ on the interval $(-\frac{1}{8}, \frac{1}{4})$.

\subsection{General case: Interpolating between $S_n^\ast$ and $S_{n+2}^\ast$}
We show that for every $x \in (1 - \frac{3}{n}, 1 - \frac{3}{n+2})$ there exists some copula $C$ with $\phi(C)=x$
fulfilling that $\rho(C) > r(x)$. 
Doing so we again consider the copulas $E_a := E_{\delta_a^\uparrow}$ from the previous subsection, work with 
finite ordinal sums and start with quickly recalling their construction (for more information 
on ordinal sums and patchworks see \cite{durante2015principles} and the reference therein).

Suppose that $n \in \mathbb{N}$, that $(a_1, b_1), (a_2,b_2),\ldots,(a_n,b_n)$ are pairwise disjoint, non-degenerated
 intervals in $[0,1]$ and that $C_1,\ldots,C_n$ are copulas. Then the ordinal sum  
 $O = \left( \langle (a_k, b_k), C_k \rangle \right)_{k=1}^n$ is the copula defined by
 \begin{align*}
        O(u,v) := 
        \begin{cases}
        a_k + (b_k - a_k) C_k \left( \frac{u-a_k}{b_k - a_k}, \frac{v-a_k}{b_k - a_k} \right), &\text{if } 
        (u,v) \in (a_k, b_k)^2 \\
        M(u,v) &\text{elsewhere.}
    \end{cases}
    \end{align*}

It is straightforward to express $\phi(O)$ and $\rho(O)$ in terms of the corresponding values of the copulas
$C_k$. 

\begin{Lemma}\label{FormulasOrdinalSum}
    Suppose that $O = \left( \langle (a_k, b_k), C_k \rangle \right)_{k=1}^n$ is the ordinal sum of $C_1,\ldots,C_n$
    with respect to $(a_1, b_1), (a_2,b_2),\ldots,(a_n,b_n)$ . Then the following formulas hold:
    \begin{align*}
        \rho(O) &= 1 - \sum_{k=1}^N (b_k - a_k)^3 (1 - \rho(C_k)) \\
        \phi(O) &= \sum_{k = 1}^N \left( 6 a_k (b_k - a_k) + (b_k - a_k)^2(\phi(C_k) + 2) \right) - 2 
    \end{align*}
\end{Lemma}

\begin{proof}
    The first identity was already shown in \cite{bukovvsek2024exact}. Concerning the second one, using change of coordinates, 
    yields
    \begin{align*}
        \frac{\phi(C) + 2}{6} &= \int_{[0,1]} C(u,u) d \lambda(u) \\&=
         \sum_{k=1}^N a_k(b_k - a_k) +  \int_{[a_k, b_k]} (b_k - a_k) C_k \left( \frac{u-a_k}{b_k - a_k}, \frac{u-a_k}{b_k - a_k} \right) d \lambda(u) \\&=
         \sum_{k=1}^N a_k(b_k - a_k) + (b_k - a_k)^2 \int_{[0, 1]}  C_k (u,u) d \lambda(u) \\&=
         \sum_{k=1}^N a_k(b_k - a_k) + (b_k - a_k)^2 \frac{\phi(C_k) + 2}{6}.
    \end{align*}
    This completes the proof.
   \end{proof}

Suppose now that $n \geq 4$ is even and set $N = n/2$. Then according to Lemma \ref{FormulasOrdinalSum} for the 
ordinal sum $O_a^N := \left( \langle ( \frac{k-1}{N},  \frac{k}{N}), E_a \rangle \right)_{k=1}^N$ we get  
\begin{align*}
    \rho(O_a^N) &= 1 - \frac{1}{N^2} (1 - \rho(E_a)) \\
    \phi(O_a^N) &= 1 - \frac{1}{N} (1 - \phi(E_a))
\end{align*}
In order to `interpolate' between $S_n^\ast$ and $S_{n+2}^\ast$ we will work with the ordinal sum 
$O_{a_N}^N$ with $a_N := \frac{N}{2N+2} \in (\frac{1}{4}, \frac{1}{2})$. The following result holds: 

\begin{Lemma}\label{PhiRhoRegionImprovement}
    For every even $n\geq4$ we have $\rho(O_{a_N}^N) > r(\phi(O_{a_N}^N))$.
\end{Lemma}

\begin{proof}
    As before we set $N = n/2$. 
    According to \cite{bukovvsek2024exact}, on the compact interval $[1 - \frac{3}{2N}, 1 - \frac{3}{2(N+1)}]$ the 
        function $r$ is given by 
        $r(x) = \frac{2N+1}{N^2 + N} x + \frac{2N^2 - 2N + 1}{2(N^2+N)}$.  \\
    Considering 
    \begin{align*}
        \phi(O_{a_N}^N) &= 1 - \frac{1}{N} \left( 1 - (6a_N^2 - 6a_N + 1)  \right) \\&=
        1 - \frac{6}{N}a_N(1-a_N) = \frac{2N^2 + N - 4}{2(N+1)^2}
    \end{align*}
    it follows that
    \begin{align*}
        1 - \frac{3}{2(N+1)} - \phi(O_{a_N}^N) = \frac{3}{2(N + 1)^2} \geq 0
    \end{align*}
    and
    \begin{align*}
        \phi(O_{a_N}^N) - (1 - \frac{3}{2N}) = \frac{3}{2N(N + 1)^2} \geq 0,
    \end{align*}
    which directly yields $ \phi(O_{a_N}^N) \in [1 - \frac{3}{2N}, 1 - \frac{3}{2(N+1)}]$. 
    Furthermore a straightforward calculation shows 
    \begin{align*}
        r(\phi(O_{a_N}^N)) &= \frac{2N^4 + 6N^3 + 3N^2 - 7N - 3}{2N(N+1)^3} \\
        \rho(O_{a_N}^N) &= \frac{2N^5 + 6N^4 + 3N^3 - 7N^2 - 3N + 1}{2N^2(N+1)^3}
    \end{align*}
    implying $\rho(O_{a_N}^N) - r(\phi(O_{a_N}^N)) = \frac{1}{2N^2(N+1)^3} > 0$. 
\end{proof}

\noindent Summing up, and using convexity of $\Omega_{\phi,\rho}$ 
we have shown the following result for the function $s:[-\frac{1}{2}, 1] \to [-1,1]$, defined by
\begin{align*}
    s(x) = 
    \begin{cases}
        2x + 1/2 - \frac{\sqrt{3}}{9} (1+2x)^{3/2}, &\text{if } x \in [-\frac{1}{2}, -\frac{1}{8}], \\[5pt]
        x + \frac{3}{8} - \frac{\sqrt{6}}{36} (1-4x)^{3/2}, &\text{if } x \in [-\frac{1}{8}, \frac{1}{4}], \\[5pt]
        \frac{\rho(O_{a_N}^N) - \rho(S_{2N}^\ast)}{\phi(O_{a_N}^N) - \phi(S_{2N}^\ast)} (x - \phi(S_{2N}^\ast)) + \rho(S_{2N}^\ast) , &\text{if } x \in [\phi(S_{2N}^\ast), \phi(O_{a_N}^N)] \text{ }, N \geq 2, \\[7.5pt]
        \frac{\rho(S_{2N+2}^\ast) - \rho(O_{a_N}^N)}{\phi(S_{2N+2}^\ast) - \phi(O_{a_N}^N)} (x - \phi(O_{a_N}^N)) + \rho(O_{a_N}^N), &\text{if } x \in [\phi(O_{a_N}^N), \phi(S_{2N+2}^\ast)] \text{ }, N, \geq 2\\[7.5pt]
        1, &\text{ if } x = 1.
    \end{cases}
\end{align*}
\begin{Theorem}
For every $x \in (-\frac{1}{8},1)$ with $x \not \in \{1-\frac{3}{2N}: N \geq 2\}$ the function $s$ fulfills 
$s(x)>r(x)$ and there exists some copula $C_x$ fulfilling $\phi(C_x)=x$ and $\rho(C_x)=s(x)$. 
\end{Theorem}

We conjecture that the upper bound $s$ is not best possible, i.e., that for every 
$x \in (-\frac{1}{8},1)$ with $x \not \in \{1-\frac{3}{2N}: N \geq 2\}$ there exists some copula 
$B_x$ fulfilling $\phi(B_x)=x$ and $\rho(B_x)>s(x)$, so deriving a globally sharp upper inequality
is still an open question.

\section*{Acknowledgements}
The first and the second author gratefully acknowledge the financial support from AMAG Austria Metall AG within the 
project 'ProSa'. The third author gratefully acknowledges the support of the WISS 2025 project `IDA-lab Salzburg'
 (20204-WISS/225/197-2019 and 20102-F1901166-KZP).





\clearpage

%\appendix

\section*{Appendix: Complementary calculations for Section \ref{SecImprovementUpperBound}}
In the sequel we derive the formulas for 
$\phi(E_{\delta_b^\downarrow})$ and $\rho(E_{\delta_b^\downarrow})$. The definition of $\delta_b^\downarrow$ implies
\begin{align*}
    \int_{[0,1]} \delta_b^\downarrow(u) d\lambda(u) &= \int_{[1/4, 1/4+b]} 2( u - \tfrac{1}{4}) d\lambda(u) + \int_{[1/4+b, 3/4 - b]} u + b - \tfrac{1}{4} d \lambda(u) \\& + \int_{[3/4 - b, 3/4]} \tfrac{1}{2} d\lambda(u) + \int_{[3/4, 1]} 2 u - 1 d \lambda(u) \\&= b^2 + \tfrac{1}{8} - 2b^2 + \tfrac{1}{2} b + \tfrac{3}{16},
\end{align*}
so $\phi(E_{\delta_b^\downarrow}) = -6 b^2 + 3b - \tfrac{1}{8}$. The diagonal copula $E_{\delta_b^\downarrow}$ distributes its mass on the function $h_b$ depicted in Figure \ref{AppendixInterpolation}, we therefore  obtain
\begin{align*}
    \int_{[0,1]} \Pi(u, h_b(u)) d\lambda(u) &= 
    \int_{[0,b]} u(u+\tfrac{1}{4}) d\lambda(u) + \int_{[b,1/4]} u(2u + \tfrac{1}{4} - b) d \lambda(u) 
    \\&\hspace{-2.5cm} + \int_{[1/4, 1/4+b]} u(u-\tfrac{1}{4}) d \lambda(u) 
    + \tfrac{1}{2} \int_{[1/4+b, 3/4-b]} u(\tfrac{1}{2}u + \tfrac{5}{8} - \tfrac{1}{2}b) d\lambda(u) 
    \\&\hspace{-2.5cm} + \tfrac{1}{2} \int_{[1/4+b, 3/4-b]} u(\tfrac{1}{2}u - \tfrac{1}{8} + \tfrac{1}{2}b) d\lambda(u) +
    \int_{[3/4-b, 3/4]} u(u+\tfrac{1}{4}) d\lambda(u)
    \\&\hspace{-1cm} + \int_{[3/4, 1 -b]} u (2u - \tfrac{5}{4} + b) d\lambda(u) + \int_{[1-b, 1]} u(u-\tfrac{1}{4}) d\lambda(u) \\&=
    \tfrac{1}{12} b^2 (8b +3) + \tfrac{1}{192}(-64b^3 - 48b^2 - 12b + 7) + \\&\hspace{-1cm} \tfrac{1}{192}(-64b^3 + 144b^2 - 204b + 43) 
    + \tfrac{1}{12} b (8b^2 - 21b + 18) 
    \\&= \tfrac{1}{96}( 64b^3 - 96b^2 + 36b + 25).
\end{align*}
Altogether, this yields $\rho(E_{\delta_b^\downarrow}) = 8b^3 - 12b^2 + \tfrac{9}{2}b + \tfrac{1}{8}$. 
Finally, considering $\tfrac{8}{3} \phi(E_{\delta_b^\downarrow}) - \tfrac{2}{3} = - (1 - 4b)^2$ and 
$\tfrac{1}{8} \left( \tfrac{2}{3} - \tfrac{8}{3} \phi(E_{\delta_b^\downarrow}) \right)^{3/2} = \phi(E_{\delta_b^\downarrow}) - \rho(E_{\delta_b^\downarrow}) + \tfrac{3}{8}$ it follows that 
\begin{align*}
    \rho(E_{\delta_b^\downarrow}) &= \phi(E_{\delta_b^\downarrow}) + \tfrac{3}{8} - \tfrac{1}{8} \left( \tfrac{2}{3} - \tfrac{8}{3} \phi(E_{\delta_b^\downarrow}) \right)^{3/2} \\&= \phi(E_{\delta_b^\downarrow}) + \tfrac{3}{8} - \tfrac{\sqrt{6}}{36} \left( 1 - 4 \phi(E_{\delta_b^\downarrow}) \right)^{3/2}.
\end{align*}

\begin{figure}[H]
    \centering
    \begin{tikzpicture}[scale=9]
    	\draw[-,line width=1] (0,0) -- (0,1) -- (1,1) -- (1,0) -- (0,0);
            
            % Support transformation
            \draw[-,line width=1, black] (0,1/4) -- (1/12, 4/12) -- (1/4, 8/12);
            \draw[-,line width=1, black] (1/4,0) -- (4/12, 1/12) -- (8/12, 1/4);
            \draw[-,line width=1, black] (4/12,9/12) -- (8/12,11/12) -- (9/12, 1);
            \draw[-,line width=1, black] (9/12,4/12) -- (11/12,8/12) -- (1,9/12);
            
            % x-Axis label
            \draw[-,line width=1] (1/12,0) -- (1/12,-0.02);
	    \node[below=1pt of {(1/12,-0.02)}, scale= 0.75, outer sep=2pt,fill=white] {$b$};
            \draw[-,line width=1] (1/4,0) -- (1/4,-0.02);
	    \node[below=1pt of {(1/4,-0.02)}, scale= 0.75, outer sep=2pt] {$\frac{1}{4}$};
            \draw[-,line width=1] (4/12,0) -- (4/12,-0.02);
	    \node[below=1pt of {(8/24,-0.02)}, scale= 0.75, outer sep=2pt] {$\frac{1}{4} + b$};
            \draw[-,line width=1] (8/12,0) -- (8/12,-0.02);
	    \node[below=1pt of {(8/12,-0.02)}, scale= 0.75, outer sep=2pt] {$\frac{3}{4} - b$};
            \draw[-,line width=1] (3/4,0) -- (3/4,-0.02);
            \node[below=1pt of {(3/4,-0.02)}, scale= 0.75, outer sep=2pt] {$\frac{3}{4}$};
            \draw[-,line width=1] (11/12,0) -- (11/12,-0.02);
	    \node[below=1pt of {(11/12,-0.02)}, scale= 0.75, outer sep=2pt] {$1 - b$};
            
            % y-Axis label
            \draw[-,line width=1] (0,1/4) -- (-0.02,1/4);
	    \node[left=1pt of {(-0.02,1/4)}, scale= 0.75, outer sep=2pt,fill=white] {$\frac{1}{4}$};
            \draw[-,line width=1] (0,4/12) -- (-0.02,4/12);
	    \node[left=1pt of {(-0.02,4/12)}, scale= 0.75, outer sep=2pt,fill=white] {$\frac{1}{4} + b$};
            \draw[-,line width=1] (0,8/12) -- (-0.02,8/12);
	    \node[left=1pt of {(-0.02,8/12)}, scale= 0.75, outer sep=2pt,fill=white] {$\frac{3}{4} - b$};
            \draw[-,line width=1] (0,9/12) -- (-0.02,9/12);
	    \node[left=1pt of {(-0.02,9/12)}, scale= 0.75, outer sep=2pt,fill=white] {$\frac{3}{4}$};
            \draw[-,line width=1] (0,11/12) -- (-0.02,11/12);
	    \node[left=1pt of {(-0.02,11/12)}, scale= 0.75, outer sep=2pt,fill=white] {$1 - b$};

            % Geradengleichungen
            \node[below=1pt of {(1/24,7/24)}, scale= 0.75, outer sep=2pt,fill=white, rotate=45] {$u + \frac{1}{4}$};
            \node[below=1pt of {(2/12,6/12)}, scale= 0.75, outer sep=2pt,fill=white, rotate=61] {$2u + \frac{1}{4} - b$};
            \node[below=1pt of {(11/48,5/48)}, scale= 0.75, outer sep=2pt,fill=white, rotate=45] {$u - \frac{1}{4}$};
            \node[below=1pt of {(21/48,11/48)}, scale= 0.75, outer sep=2pt,fill=white, rotate=27] {$\frac{1}{2}u - \frac{1}{8} + \frac{1}{2} b$};
            \node[below=1pt of {(9/12,6/12)}, scale= 0.75, outer sep=2pt,fill=white, rotate=61] {$2u - \frac{5}{4} + b$};
            \node[below=1pt of {(21/48,77/96)}, scale= 0.75, outer sep=2pt,fill=white, rotate=27] {$\frac{1}{2}u + \frac{5}{8} - \frac{1}{2} b$};
            \node[below=1pt of {(87/96,37/48)}, scale= 0.75, outer sep=2pt,fill=white, rotate=45] {$u - \frac{1}{4}$};
            \node[below=1pt of {(69/96,93/96)}, scale= 0.75, outer sep=2pt,fill=white, rotate=45] {$u + \frac{1}{4}$};
	\end{tikzpicture}
    \caption{Mass distribution of the diagonal copula $E_{\delta_b^\downarrow}$ on the graph of $h_b$.}\label{AppendixInterpolation}
\end{figure}
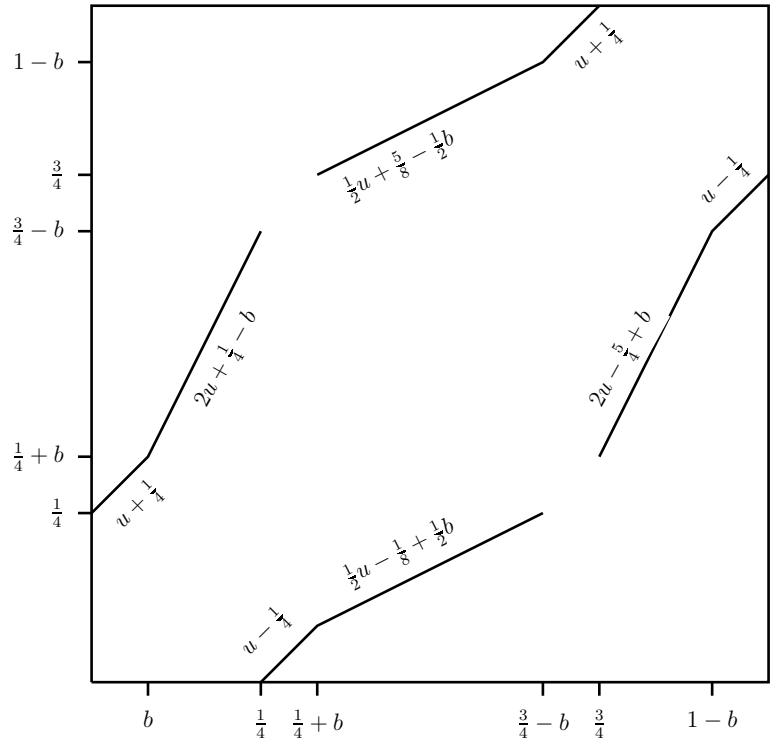

\end{document}